\documentclass{article}

\usepackage{amsfonts}
\usepackage{amsmath}
\usepackage{amssymb}
\usepackage{mathtools}
\usepackage[T1]{fontenc}
\usepackage{mathrsfs}
\usepackage{bbm} %
\usepackage{booktabs}
\makeatletter
\renewcommand*\env@matrix[1][\arraystretch]{%
  \edef\arraystretch{#1}%
  \hskip -\arraycolsep
  \let\@ifnextchar\new@ifnextchar
  \array{*\c@MaxMatrixCols c}}
\makeatother
\usepackage{pgfplots}
\pgfplotsset{compat=1.18}
\usepackage{upgreek}
\usepackage{graphicx}
\usepackage{dsfont}
\usepackage[english]{babel}
\usepackage[utf8]{inputenc}
\usepackage{mleftright}
\usepackage[normalem]{ulem}
\usepackage{multirow}
\usepackage[paper=a4paper,margin=1in]{geometry}
\usepackage{amsthm}
\usepackage[nottoc]{tocbibind}
\usepackage[numbers,sort]{natbib}
\usepackage{hyperref}
\usepackage[nameinlink]{cleveref}
\expandafter\def\csname ver@etex.sty\endcsname{3000/12/31}

\newcommand{\unitCardFunc}[1]{\mathbbm{I}\mleft(#1 \mright)}
\newcommand{\conclusionSectTitle}{Conclusions and future work}
\newcommand{\setFunct}[2]{\left\{ {#1} \, : \, {#2} \right\}}
\newcommand{\linearConSetSymbol}{\mathcal{F}}

\usepackage{scalerel}
\usepackage{tikz}
\usetikzlibrary{svg.path}

\definecolor{orcidlogocol}{HTML}{A6CE39}
\tikzset{
  orcidlogo/.pic={
    \fill[orcidlogocol] svg{M256,128c0,70.7-57.3,128-128,128C57.3,256,0,198.7,0,128C0,57.3,57.3,0,128,0C198.7,0,256,57.3,256,128z};
    \fill[white] svg{M86.3,186.2H70.9V79.1h15.4v48.4V186.2z}
                 svg{M108.9,79.1h41.6c39.6,0,57,28.3,57,53.6c0,27.5-21.5,53.6-56.8,53.6h-41.8V79.1z M124.3,172.4h24.5c34.9,0,42.9-26.5,42.9-39.7c0-21.5-13.7-39.7-43.7-39.7h-23.7V172.4z}
                 svg{M88.7,56.8c0,5.5-4.5,10.1-10.1,10.1c-5.6,0-10.1-4.6-10.1-10.1c0-5.6,4.5-10.1,10.1-10.1C84.2,46.7,88.7,51.3,88.7,56.8z};
  }
}

\newcommand\orcid[1]{\,\href{https://orcid.org/#1}{\mbox{\scalerel*{\hspace{1cm}
\begin{tikzpicture}[yscale=-1,transform shape]
\pic{orcidlogo};
\end{tikzpicture}
}{|}}}}

\newcommand{\linearConSet}[1]{\linearConSetSymbol\mleft( {#1} \mright)}

\newcommand{\Tr}[1]{\text{tr}\left( {#1}\right)}
\DeclarePairedDelimiter\abs{\lvert}{\rvert}

\DeclareMathOperator*{\argmin}{arg\,min}
\newcommand{\N}{\mathbb{N}}
\newcommand{\R}{\mathbb{R}}
\newcommand{\lassColName}{SDP-Lasserre}

\newcommand{\eigenDecomp}{eigenvalue decomposition}
\newcommand{\eigenDecompCap}{Eigenvalue decomposition}

\newcommand{\NpComplexity}{NP}
\newcommand{\PComplexity}{P}
\newcommand{\colName}{GapClsd}
\newcommand{\stepSize}{\nu}
\newcommand{\indicatorVec}[1]{\mathbbm{1}_{#1}}

\newcommand{\symmMatSet}{\mathcal{S}}

\newcommand{\allStable}[1]{\left[n \right]_{#1}}
\newcommand{\edgeIdeal}[1]{\mathcal{I}_{#1}}
\newcommand{\edge}[2]{\left\{#1,#2\right\}}

\newcommand{\pairwiseUnion}[1]{{#1}^{2 \cup}}

\newcommand{\colNameEigs}{Bnd.$\!$ diff.}

\newcommand{\LassHierarc}{Lasserre hierarchy}

\newcommand{\posOnG}[1]{\mathbb{P}_{#1}}
\newcommand{\lassObj}[1]{v\mleft(#1\mright)}
\newcommand{\posOnGB}[2]{\mathbb{P}_{#1}\mleft( {#2} \mright)}
\newcommand{\partialPowerSet}[2]{\binom{\left[#1 \right]}{\leq {#2}}}
\newcommand{\spec}[1]{\Lambda\mleft({#1}\mright)}
       
\date{}

\newtheorem{lemma}{Lemma}

\begin{document}

\title{SDP bounds on the stability number via ADMM and intermediate levels of the Lasserre hierarchy} 

\author{Lennart Sinjorgo\thanks{CentER,  Department of Econometrics and OR, Tilburg University, The Netherlands \\
\hphantom{asd} \texttt{\{l.m.sinjorgo,r.sotirov,j.c.veralizcano\}@tilburguniversity.edu}}\,\, \orcid{0000-0003-0516-6360} \and Renata Sotirov\thanks{corresponding author}\footnotemark[1] \orcid{0000-0002-3298-7255} 
\and Juan C. Vera\footnotemark[1] \orcid{0000-0003-1394-3422}}

\maketitle
\vspace{-1em}
\begin{abstract}
We consider the Lasserre hierarchy for computing bounds on the stability number of graphs. The semidefinite programs (SDPs) arising from this hierarchy involve large matrix variables and many linear constraints, which makes them difficult to solve using interior-point methods. We propose solving these SDPs using the alternating direction method of multipliers (ADMM). When the second level of the Lasserre hierarchy for a given graph is intractable for the ADMM, we consider an intermediate-level relaxation of the hierarchy. To warm-start the ADMM, we use an optimal solution from the first level of the Lasserre hierarchy, which is equivalent to the well-known Lov\'asz theta function. Additionally, we use this solution to determine which degree two monomials to add in the Lasserre hierarchy relaxation to obtain an intermediate level between 1 and 2. Computational results demonstrate that our approach yields strong bounds on the stability number, which are computable within reasonable running times. 
We provide {improved} bounds on the stability number of various graphs from the literature.
\end{abstract}

{\bf Keywords.} semidefinite programming, stability number, ADMM, Lasserre hierarchy

\section{Introduction}
\label{section_Intro1}
A stable set of a graph is a subset of vertices that are pairwise non-adjacent. 
Given an undirected graph~$G$, the stable set problem is to determine a stable set of $G$ of maximum cardinality. 
The stability number of~$G$, denoted by $\alpha(G)$, is defined as the cardinality of a maximum stable set in $G$. Computing $\alpha(G)$ is \NpComplexity{}-hard. Hence, unless \PComplexity{} $=$ \NpComplexity{}, there is no polynomial time algorithm that computes it.

There exist different approaches for computing upper bounds on the stability number of a graph, and one of those is using semidefinite programming.
The first  SDP relaxation of the stable set problem is due to \citeauthor{lovasz1979shannon} \cite{lovasz1979shannon}, who introduced the Lov{\'a}sz theta function of a graph $G$, denoted by $\vartheta(G)$. For any graph $G$, $\vartheta(G) \geq \alpha(G)$ and $\vartheta(G)$ can be computed in polynomial time up to fixed precision. 
Semidefinite programs (SDPs) that define the Lov{\'a}sz theta function can be strengthened by cutting planes, see e.g., \cite{dukanovic2007semidefinite,gruberStableSet,schrijver1979comparison,pucher2024class}. 
The paper by Pucher and Rendl \cite{pucher2024class} currently provides one of the strongest SDP-based bounds for the stable set problem.

Several hierarchical approaches can also be applied to construct relaxations of the stable set problem.
Higher levels in the hierarchy correspond to stronger relaxations, which are also more difficult to solve due to the increased number of variables and constraints.  
Among the hierarchies that can be applied to the stable set problem are the Sherali-Adams hierarchy~\cite{sherali1990hierarchy} (based on  linear programming), the SDP-based hierarchy of Lov\'asz-Schrijver~\cite{lovasz1991cones}, the Lasserre hierarchy~\cite{lasserre2001global}, and the exact subgraph hierarchy (ESH)~\cite{adams2015hierarchy}.
Much research has recently been devoted to the ESH for the stable set problem~\cite{gaar2020computational,gaar2022sdp,gaar2024different}. 
The computational results in those papers show that the ESH  provides strong bounds within a reasonable computational time. 
It is known that the Lasserre hierarchy is stronger than the other hierarchies~\cite{laurent2003comparison,silvestri2013spectrahedral}. 
Despite this, little research has been devoted to the practical performance of the Lasserre hierarchy for the stable set problem.  A practical drawback of the Lasserre hierarchy is the order of the associated  positive semidefinite (PSD) matrix variable: level $k$ of the hierarchy involves a PSD matrix variable of order $\mathcal{O}(n^k)$, where $n$ is the number of vertices in the graph.

The classical method for solving semidefinite programs (SDPs) is the interior-point method (IPM)~\cite{NesterovNemirovskii,Alizadeh}. However, IPMs typically require large amounts of memory,  which limits their applicability to the Lasserre hierarchy.
The interior-point method is a second-order method that requires the construction and  Cholesky factorization of an   $m \times m$ Schur complement matrix, where $m$ is the number of linear equality constraints in the SDP relaxation.   Constructing and storing this dense matrix requires substantial memory,  particularly  when $m$ is large, such as in SDPs  derived from the Lasserre hierarchy. 
The worst-case complexity of computing the  Schur complement matrix is  ${\mathcal O}( mp^3+m^2p^2)$~\cite{Monteiro01011999}, where $p$ is the order of the PSD variable.
Computational complexity may be reduced in cases where the constraint matrices exhibit special structures such as low rank, see e.g.,~\cite{Ivanov01062010,Klerk2010ExploitingSS}. The Cholesky factorization for the Schur complement matrix requires ${\mathcal O}(m^3)$ operations,  which becomes impractical for large values of~$m$. These limitations of IPMs motivate the use of alternative methods that require less memory and bypass the Cholesky factorization.

The alternating direction method of multipliers (ADMM), see e.g.,~\cite{boyd2011distributed}, is a first-order method that can be used to solve SDPs, and requires significantly less memory than IPMs. Each iteration of the ADMM algorithm for solving SDPs consists of three steps: the orthogonal  projection onto the cone of positive semidefinite matrices, an orthogonal projection onto a polyhedral set, and a dual update step.
The most memory intensive step of the ADMM  is computing the \eigenDecomp{} of a symmetric PSD  matrix of order $p$, in each main loop of the algorithm.
The SDP relaxations of the stable set problem arising from the Lasserre hierarchy satisfy  $m \gg p$, and the computational complexity of \eigenDecomp{} for a symmetric matrix of order $p$ is ${\mathcal O}(p^3)$. Considering this, along with the fact that projections onto the polyhedral sets from the Lasserre relaxations can be performed efficiently, ADMMs appear more suitable than IPMs for computing Lasserre bounds for the stable set problem.

{ In the literature there exist methods designed for solving SDPs that require low per-iteration computational cost and therefore suited for solving large scale SDPs. We list some of them.
An augmented Lagrangian-based solver SDPNAL+~\cite{Sun2020SDPNAL} for solving large scale SDPs with bound constraints.
The conic operator splitting method COSMO~\cite{Garstka_2021} suited for sparse and structured problems, and ADMM-based algorithm for sparse SDPs  CDCS \cite{Zheng2020}.
A conditional gradient-based augmented Lagrangian method CGAL~\cite{Yurtsever2019-Cgal},
and its extended variant, SketchyCGAL \cite{Yurtsever2021-ScalableSDP}, which improves the scalability of CGAL by incorporating randomized matrix sketching. 
Loraine~\cite{Habibi2025} is an interior-point method that uses a preconditioned
conjugate gradient method to avoid forming the Schur complement matrix, and is
designed for SDPs with very low-rank solutions.
The SDPs considered in this paper have a structure that favours the ADMM.
First, the matrix variable is of moderate order (at most $2500$), whereas the number
of constraints is very large; each entry of the matrix variable appears in at most
one constraint, so that the projection onto the polyhedral part of the feasible set
is available in closed form (see Lemma~\ref{eqn_minimizationProb1}).
Second, every PSD iterate of the ADMM yields an upper bound on $\alpha(G)$
(see Lemma~\ref{eqn_lemmaShiftState}); hence high-accuracy solutions are not required, and a valid bound is available at any time.
These two properties make the main advantages of the methods above less relevant
in our setting: the matrix variable is small enough to be stored, which removes the
benefit of sketching; the aggregate sparsity pattern of our SDPs is (nearly) dense,
which prevents chordal decomposition; and the optimal solutions of the SOS
formulation~\eqref{eqn_alphaBCompact}-\eqref{eqn_sdpCons}  are in general not of low rank, which is the regime targeted
by SketchyCGAL and Loraine.
Moreover, the bound of Lemma~\ref{eqn_lemmaShiftState} penalizes the total constraint violation,
so methods whose iterates converge slowly to feasibility, such as CGAL-type methods,
would yield weaker bounds than the ADMM within the same running time.}

In this paper, we bridge the gap between theory and practice by using the ADMM to effectively compute
Lasserre hierarchy bounds for the stability number.
In particular, we compute bounds from intermediate levels of the Lasserre hierarchy for  $k$ between $1$ and $2$, including  $k=2$, on graphs with at most 300 vertices. 
However, we are not the first to consider intermediate levels of the Lasserre hierarchy; these have been employed in several other papers, see e.g.,~\cite{Campos2022PartialLR,chen2022sublevel,sinjorgo2024solving,van2008sums,anjos2005improved}.
For the majority of graphs, we restrict ourselves to intermediate levels of the hierarchy due to practical limitations on the order of the considered PSD variable, which we cap at 2500. \eigenDecompCap for a symmetric matrix of that order can still be performed reasonably well, especially when single precision is used.
The number of inequality constraints in the resulting SDP relaxation depends on the  considered graph and is at most $2,510,148$ in our experiments.
Storing the $m \times m$ Schur complement matrix, with $m = 2,510,148$, requires $46,944.9$ GB of RAM (!). 
Therefore, IPMs are intractable for solving the corresponding relaxations. 
Constructing an intermediate-level SDP relaxation of the Lasserre hierarchy for $k$ between 1 and 2 requires selecting specific degree two monomials.
These monomials of degree two, along with all monomials of degree one and the monomial of degree zero, form a basis used to derive the SDP relaxation.
We present a basis selection method  that exploits an SDP relaxation of the Lov\'asz theta function. 
It is known that $\vartheta(G)$  corresponds to the first level of the Lasserre hierarchy applied to the stable set problem, see e.g., \cite[Sect.~6]{laurent2003comparison}. Our ADMM algorithm incorporates a warm-starting approach to further improve performance.
{
The computational results show that the upper bounds on $\alpha(G)$ computed here are competitive with the SDP-based bounds reported in the literature.}
Moreover, these bounds can be obtained using the ADMM within reasonable running times, specifically within one hour.\\

This paper is organized as follows. The notation is introduced in  \Cref{section_notation}. We provide details of the Lasserre hierarchy for the stable set problem in \Cref{section_lassHierarc}. In \Cref{section_ADMM}, we show how to apply the ADMM to the SDPs arising from the Lasserre hierarchy. In \Cref{section_dynBasisWarmStart} we present a basis selection method for the construction of intermediate levels of the Lasserre hierarchy, and provide a method for warm-starting the ADMM. Computational  results are presented in \Cref{section_numericalResults}, and the conclusions are provided in \Cref{section_conclusion}.

\subsection{Notation}
\label{section_notation}
Given $n,k \in \mathbb{N}$, we define $[n] := \{1,\dots,n\}$ and $\partialPowerSet{n}{k} := \setFunct{ \beta \subseteq [n]}{ | \beta | \leq k}$. Let $\mathcal{B} \subseteq \partialPowerSet{n}{k}$. We define $\pairwiseUnion{\mathcal{B}}:= \setFunct{\beta \cup \beta^\prime}{\beta,\beta^\prime \in \mathcal{B}}$. For any $\beta \subseteq [n]$, we define $\indicatorVec{\beta} \in \{0,1\}^n$ as the indicator vector corresponding to $\beta$, i.e., $(\indicatorVec{\beta})_i = 1$ if and only if $i \in \beta$. We define $\unitCardFunc{\beta} \in \{0,1\}$ as the indicator that equals $1$ if $\abs{\beta} = 1$, and $0$ otherwise. We define $\mathbb{R}^k_+$ as the set of entrywise nonnegative $k$-dimensional real vectors. We denote by $\symmMatSet^k$  the set of symmetric matrices of order $k$, and by $\symmMatSet^k_+$  the set of symmetric PSD matrices of order $k$. We write $A \succeq 0$ to indicate that $A \in \symmMatSet^k_+$. For $a \in \mathbb{R}^k$, we define $\text{Diag}(a) \in \symmMatSet^k$ as the diagonal matrix with $a$ on its diagonal. For $A,B \in \symmMatSet^{k}$, we define the inner product $\langle A, B \rangle := \Tr{AB} = \sum_{i = 1}^k \sum_{j=1}^k A_{ij} B_{ij}$. The Frobenius norm of $A \in \symmMatSet^k$ is defined as $\| A \| := \sqrt{\langle A, A \rangle}$. For any closed convex $W \subseteq \symmMatSet^k$, the projection of $A \in \symmMatSet^k$ onto $W$ is defined as $\mathcal{P}_{W}(A) := \argmin_{X \in W} \| X - A \|$. 

Define $\mathbb{R}[x]$ as the set of polynomials in the variables $x_1, \dots, x_n$ with real coefficients. Given  $\mathcal{B} \subseteq \partialPowerSet{n}{k}$, we define $\mathbf{x}^\mathcal{B}$ as the $\abs*{\mathcal{B}}$-dimensional vector of all monomials $x^\beta := \prod_{i \in \beta} x_i$, $\beta \in \mathcal{B}$, with $x^\emptyset = 1$. The monomials $(x^\beta)_{\beta \in \mathcal{B}}$ form a basis of some subspace of $\mathbb{R}[x]$. With slight abuse of terminology, we refer to both $(x^\beta)_{\beta \in \mathcal{B}}$ and $\mathcal{B}$ as a bases.

\section{The Lasserre hierarchy for the stable set problem}
\label{section_lassHierarc}
We present the Lasserre hierarchy \cite{lasserre2001global} for the stable set problem. 
Similar derivations can also be found in, e.g., \cite[Sect.~3.1]{gouveia2010theta}, \cite{laurent2007semidefinite}, and \cite[Example~8.16]{laurent2009sums}.

Let $G = (V,E)$ be a simple undirected graph. Without loss of generality, we assume that $V = [n]$ for some $n \in \mathbb{N}$. Let $\beta \subseteq [n]$. If $\beta$ is a stable set in $G$, we say that $\beta$ is stable in $G$.
We define \mbox{$\allStable{G} := \setFunct{ \beta \subseteq [n]}{\beta \text{ is stable in }G}$} as the set of all stable sets of $G$, and set $S_G := \setFunct{ \indicatorVec{\beta}}{\beta \in \allStable{G}}$. We define $\posOnG{G} \subseteq \mathbb{R}[x]$ as the set of polynomials nonnegative over $S_G$.

Observe that the stability number $\alpha(G) = \max_{x \in S_G } \sum_{i \in [n]} x_i$, which is equivalent to
\begin{align}
    \label{eqn_popReform}
    \alpha(G) = \min_{\mu \in \mathbb{R}} \setFunct{ \mu }{  \mu - \sum_{i \in [n]} x_i  \in \posOnG{G}}.
\end{align}
This equivalence follows from the observation that for fixed $\mu$ we have $\min_{x \in S_G} \mu - \sum_{i \in [n]} x_i = \mu - \max_{x \in S_G} \sum_{i \in [n]} x_i = \mu - \alpha(G)$, which implies that $\mu - \sum_{i \in [n]} x_i$ is a nonnegative polynomial over $S_G$ if and only if $\mu \geq \alpha(G)$. In general, it is \NpComplexity{}-hard to optimize over $S_G$ or $\posOnG{G}$, which motivates the search of tractable relaxations of \eqref{eqn_popReform}. Let us formulate such a relaxation, in terms of sum of squares (SOS) polynomials. 
To this end,
we define the polynomial ideal
\begin{align}
    \label{eqn_edgeIdeal}
    \edgeIdeal{G} := \left\langle x_i^2 - x_i \text{ for all } i \in [n], \, x_i x_j \text{ for all } \edge{i}{j} \in E\right\rangle,
\end{align}
that is used to define $\posOnG{G}$ as follows:
\begin{align*}
    \posOnG{G} := \setFunct{f \in \mathbb{R}[x]}{f \equiv \sum_{j \in [k]} f^2_j \mod \edgeIdeal{G}, \, f_j \in \mathbb{R}[x] \text{ for all } j \in [k], \, k \in \mathbb{N}},
\end{align*}
see e.g., \cite[Thm.~1]{parrilo2002explicit}.
Note that $\edgeIdeal{G}$ encodes the set $S_G$ in the sense that $x \in S_G$ if and only if $f(x) = 0$ for all $f \in \edgeIdeal{G}$. 
Polynomials of the form $\sum_{j \in [k]} f^2_j$ are called SOS polynomials. SOS polynomials, and thus polynomials in $\posOnG{G}$, can be expressed in terms of PSD matrices. Specifically,  we have that
\begin{align*}
    f \in \posOnG{G} \iff f \equiv  \left( \mathbf{x}^{\mathcal{B}'} \right)^\top A \mathbf{x}^{\mathcal{B}'} \mod \edgeIdeal{G} \text{ for some }A \in \symmMatSet^{\abs*{\mathcal{B}'}}_+,
\end{align*}
where $\mathcal{B}' := \partialPowerSet{n}{n}$, 
see e.g., \cite[Prop.~2.1]{lasserre2009moments}. This shows that one can optimize over $\posOnG{G}$ using semidefinite programming. However, $\abs*{\mathcal{B}'}$ is exponential in $n$, which makes $\posOnG{G}$ intractable. It is therefore natural to consider a subset of $\posOnG{G}$ by fixing a $\mathcal{B} \subseteq \partialPowerSet{n}{n}$ such that $\abs*{\mathcal{B}}$ is polynomial in $n$. For any $\partialPowerSet{n}{1} \subseteq \mathcal{B} \subseteq \partialPowerSet{n}{n}$, we define a corresponding subset of $\posOnG{G}$ as
\begin{align}   
\label{eqn_posB_relaxation}
    \posOnGB{G}{\mathcal{B}} := \setFunct{f \in \mathbb{R}[x]}{ f \equiv \left(\mathbf{x}^\mathcal{B}\right)^\top A  \mathbf{x}^\mathcal{B} + c^\top \mathbf{x}^{\pairwiseUnion{\mathcal{B}}} \mod \mathcal{I}_G, \, A \in \symmMatSet^{\abs*{\mathcal{B}}}_+, \,\, c \in \mathbb{R}^{\abs*{\pairwiseUnion{\mathcal{B}}}}_+}.
\end{align}
In \eqref{eqn_posB_relaxation}, the term $c^\top \mathbf{x}^{\pairwiseUnion{\mathcal{B}}}$ is introduced  to enlarge $\posOnGB{G}{\mathcal{B}}$, resulting in stronger bounds on $\alpha(G)$. Moreover, we show in \Cref{section_explicitSDP} that the addition of this term does not increase the computational cost of obtaining these bounds. Note that $c^\top \mathbf{x}^{\pairwiseUnion{\mathcal{B}}}$ is an SOS polynomial modulo $\edgeIdeal{G}$, since $c^\top \mathbf{x}^{\pairwiseUnion{\mathcal{B}}} \equiv \left( \mathbf{x}^{\pairwiseUnion{\mathcal{B}}} \right)^\top \text{Diag}(c) \, \mathbf{x}^{\pairwiseUnion{\mathcal{B}}} \mod \edgeIdeal{G}$, and $\text{Diag}(c) \succeq 0$. Since $\left(\mathbf{x}^\mathcal{B}\right)^\top A  \mathbf{x}^\mathcal{B}$ is also an SOS polynomial, it follows that $ \posOnGB{G}{\mathcal{B}} \subseteq \posOnG{G}$. From \eqref{eqn_popReform}, we observe that the value
\begin{align}
    \label{eqn_alphaBDef}
    \alpha^\mathcal{B}(G) := \min_{\mu \in \mathbb{R}} \setFunct{ \mu }{  \mu - \sum_{i \in [n]} x_i  \in \posOnGB{G}{\mathcal{B}}}
\end{align}
satisfies $\alpha^\mathcal{B}(G) \geq \alpha(G)$, since $\posOnGB{G}{\mathcal{B}} \subseteq \posOnG{G}$. Note that $\alpha^\mathcal{B}(G)$ is well-defined since $\partialPowerSet{n}{1} \subseteq \mathcal{B}$. Computing $\alpha^\mathcal{B}(G)$ is equivalent to solving an SDP (see \Cref{section_explicitSDP}) wherein the PSD variable is of order $\abs*{\mathcal{B}}$. If $\abs*{\mathcal{B}}$ is polynomial in $n$, the value $\alpha^\mathcal{B}(G)$ can be computed in polynomial time up to fixed precision \cite[Cor.~9]{raghavendraSOS}.

It is worth noting that the computational effort of computing $\alpha^\mathcal{B}(G)$ for some $\mathcal{B} \subseteq \partialPowerSet{n}{n}$ can be reduced significantly by computing instead $\alpha^{\mathcal{B} \cap  \allStable{G}}(G)$. Indeed, $\abs*{\mathcal{B} \cap  \allStable{G}} \leq \abs*{\mathcal{B}}$, which results in a smaller PSD variable, and $\alpha^\mathcal{B}(G) = \alpha^{\mathcal{B} \cap  \allStable{G}}(G)$, as the next result shows. 
\begin{lemma}
\label{lemma_sizeReduction}
    Let $G$ be a graph on $n$ vertices and $\mathcal{B} \subseteq \partialPowerSet{n}{n}$. For $\alpha^\mathcal{B}(G)$ as in \eqref{eqn_alphaBDef}, we have $\alpha^\mathcal{B}(G) = \alpha^{\mathcal{B} \cap  \allStable{G}}(G)$.
\end{lemma}
\begin{proof}
    For notational convenience, we write $\mathcal{B}' := \mathcal{B} \cap  \allStable{G}$. By the definition of $\alpha^\mathcal{B}(G)$, it suffices to show that $\posOnGB{G}{\mathcal{B'}} = \posOnGB{G}{\mathcal{B}}$. Since $\mathcal{B}' \subseteq \mathcal{B}$, it is clear from \eqref{eqn_posB_relaxation} that $\posOnGB{G}{\mathcal{B'}} \subseteq \posOnGB{G}{\mathcal{B}}$. Thus, it remains to show the reverse inclusion. Let $f \in \posOnGB{G}{\mathcal{B}}$, and let $A \in \symmMatSet^{\abs*{\mathcal{B}}}_+$ and $c \in \mathbb{R}^{\abs*{\pairwiseUnion{\mathcal{B}}}}_+$ satisfy $f \equiv \left(\mathbf{x}^\mathcal{B}\right)^\top A  \mathbf{x}^\mathcal{B} + c^\top \mathbf{x}^{\pairwiseUnion{\mathcal{B}}} \mod \mathcal{I}_G$, where $\edgeIdeal{G}$ is defined in \eqref{eqn_edgeIdeal}. Let $A$ be indexed by the elements of $\mathcal{B}$, and let $A'$ be the principal submatrix of $A$ indexed by the elements of $\mathcal{B}'$. We define similarly the vector $c' = (c_\beta)_{\beta \in \pairwiseUnion{{\left(\mathcal{B}'\right)}}}$. Since $x^\beta \equiv 0 \mod \edgeIdeal{G}$ for any $\beta \in \mathcal{B} \setminus \mathcal{B}'$, we have that
    \begin{align*}
f \equiv \left(\mathbf{x}^\mathcal{B}\right)^\top A  \mathbf{x}^\mathcal{B} + c^\top \mathbf{x}^{\pairwiseUnion{\mathcal{B}}} \equiv \left(\mathbf{x}^{\mathcal{B}'}\right)^\top A'  \mathbf{x}^{\mathcal{B}'} + \left(c' \right)^\top \mathbf{x}^{\pairwiseUnion{{\left(\mathcal{B}'\right)}}} \mod \edgeIdeal{G} \implies f \in \posOnGB{G}{\mathcal{B}'},
    \end{align*}
    which completes the proof.
\end{proof}

\newcommand{\lassBasis}{\partialPowerSet{n}{k}}
The $k$th level of the Lasserre hierarchy for the stable set problem is to compute $\alpha^{\lassBasis}(G)$, or equivalently, $\alpha^{\partialPowerSet{n}{k} \cap \allStable{G}}(G)$. For any fixed value of $k$, $\abs*{\lassBasis} \in \mathcal{O}(n^k)$, and thus, $\alpha^{\lassBasis}(G)$ can be computed in polynomial time. The sequence $(\alpha^{\lassBasis}(G))_{k \in \mathbb{N}}$ is decreasing towards $\alpha(G)$. Moreover, if $\alpha(G) \geq 2$, then $\alpha^{\lassBasis}(G) = \alpha(G)$ for $k \geq \alpha(G)-1$ \cite[Prop.~21]{laurent2003comparison}.

\section{The ADMM for computing \texorpdfstring{$\alpha^{\mathcal{B}}(G)$}{alphaB(G)}}
\label{section_ADMM}

In this section we propose the ADMM for computing $\alpha^{\mathcal{B}}(G)$, for general bases $\partialPowerSet{n}{1} \subseteq \mathcal{B} \subseteq \partialPowerSet{n}{n}$. 
The ADMM has been successfully used to solve SDPs, see e.g., \cite{rontsis2022efficient,sinjorgo2024solving,madani2015admm,oliveira2018admm}. Compared to the IPM,  the classical method for solving SDPs, the ADMM requires less memory, making it better suited for solving the large-scale SDPs that arise in the Lasserre hierarchy for the stable set problem.

Throughout the remainder of this section we fix some basis $\mathcal{B}$ that satisfies $\partialPowerSet{n}{1} \subseteq \mathcal{B} \subseteq \partialPowerSet{n}{n}$.
\subsection{The SDP defining \texorpdfstring{$\alpha^\mathcal{B}(G)$}{alphaB(G)}}
\label{section_explicitSDP}
Consider the constraint $\mu - \sum_{i \in [n]} x_i  \in \posOnGB{G}{\mathcal{B}}$ in the definition of $\alpha^\mathcal{B}(G)$, given by \eqref{eqn_alphaBDef}. It follows from \eqref{eqn_posB_relaxation} that this constraint is equivalent to
\begin{align}
    \label{eqn_polyEquiv1}
    \mu - \sum_{i \in [n]} x_i   \equiv \left( \mathbf{x}^\mathcal{B} \right)^\top A \mathbf{x}^\mathcal{B} + c^\top \mathbf{x}^{\pairwiseUnion{\mathcal{B}}}\mod \mathcal{I}_G, \, A \in \symmMatSet^{\abs*{\mathcal{B}}}_+, \, c \in \mathbb{R}^{\abs*{\pairwiseUnion{\mathcal{B}}}}_+,
\end{align}
where $\edgeIdeal{G}$ is defined in \eqref{eqn_edgeIdeal}. Let us index the matrix $A$ with the elements of $\mathcal{B}$, and $c$ with elements of $\pairwiseUnion{\mathcal{B}}$.  

The equivalence relation \eqref{eqn_polyEquiv1} implies that the two polynomials have equal coefficients of $x^\beta$, for all $\beta \in \allStable{G}$. Thus, \eqref{eqn_polyEquiv1} implies the following linear equalities: $A_{\emptyset,\emptyset} + c_\emptyset = \mu$ and 
\begin{align}    
\label{eqn_eliminateCDiff}
\sum_{\beta \in \mathcal{B}} \, \,\sum_{\beta^\prime \in \mathcal{B} : \beta \cup \beta^\prime = \gamma} A_{\beta, \beta^\prime} \, + c_\gamma = -\unitCardFunc{\gamma} \, \text{ for all } \gamma \in \pairwiseUnion{\mathcal{B}} \cap \allStable{G}, \, \, \gamma \neq \emptyset. 
\end{align}
Here, we consider $\gamma \in  \pairwiseUnion{\mathcal{B}}$ since for any $\beta,\beta' \subseteq [n]$, $x^\beta x^{\beta'} \equiv x^{\beta \cup \beta'} \mod \edgeIdeal{G}$. Since the objective in \eqref{eqn_alphaBDef} is to minimize $\mu$, it follows that at optimality, we have $A_{\emptyset, \emptyset} = \mu$ and $c_\emptyset = 0$. We eliminate the other entries of $c$, by transforming the equality constraints from \eqref{eqn_eliminateCDiff} into inequality constraints, using that $c \geq 0$. It then follows that
\begin{align}
    \label{eqn_alphaBCompact}
    \alpha^\mathcal{B}(G) = \min \setFunct{A_{\emptyset,\emptyset}}{A \in \symmMatSet^{\abs*{\mathcal{B}}}_+ \cap \linearConSet{\mathcal{B}} },
\end{align}
where
\begin{equation}
    \label{eqn_sdpCons}
    \linearConSet{\mathcal{B}} := \setFunct{ A \in \symmMatSet^{\abs*{\mathcal{B}}} }{ \sum_{\beta \in \mathcal{B}} \, \,\sum_{\beta^\prime \in \mathcal{B} : \beta \cup \beta^\prime = \gamma} A_{\beta, \beta^\prime} \leq - \unitCardFunc{\gamma} \text{ for all } \gamma \in \pairwiseUnion{\mathcal{B}} \cap \allStable{G}, \, \gamma \neq \emptyset
    }.
\end{equation}
We will use formulation \eqref{eqn_alphaBCompact} to compute $\alpha^\mathcal{B}(G)$ via the ADMM.

\subsection{The ADMM iterates}
\label{section_admmIterates}
To apply the ADMM to \eqref{eqn_alphaBCompact}, we first reformulate \eqref{eqn_alphaBCompact} as the following optimization problem:
\begin{equation}
\label{eqn_PPsubReform2}
\begin{aligned}
\min_{X, \,  Y \in \symmMatSet^{\abs*{\mathcal{B}}}} \quad & Y_{\emptyset,\emptyset} \\
    \text{s.t. } \quad &  X \in \symmMatSet^{\abs*{\mathcal{B}}}_+, \, \, Y \in \linearConSet{\mathcal{B}}, \, \, X=Y,
    \end{aligned}
\end{equation}
where $\linearConSet{\mathcal{B}}$ is defined in \eqref{eqn_sdpCons}. Given a penalty parameter $\rho > 0$, the augmented Lagrangian associated to \eqref{eqn_PPsubReform2}, with respect to the constraint $X = Y$, is the function
\begin{align} \label{eqn_augmentedLagrange}
    L_\rho\mleft(X,Y,Z \mright) := Y_{\emptyset,\emptyset} + \rho \left\langle Z, Y-X \right\rangle + \frac{\rho}{2} \left\| Y-X \right\|^2,
\end{align}
where $Z \in \symmMatSet^{\abs*{\mathcal{B}}}$ is the
(scaled) dual variable. Given some initial $X^1, Y^1, Z^1 \in \symmMatSet^{\abs*{\mathcal{B}}}$, the ADMM computes the sequence of matrices $\left( X^\ell, Y^\ell, Z^\ell \right)_{\ell \in \mathbb{N}}$, defined recursively as
\begin{equation}
\label{eqn_admmIterates}
\begin{alignedat}{2}
    &X^{\ell+1} &&:= \argmin_{X \succeq 0} L_\rho\mleft(X, Y^\ell,Z^\ell\mright) \\
    &Y^{\ell+1} && := \argmin_{Y \in \linearConSet{\mathcal{B}}} L_\rho \mleft(X^{\ell+1},Y,Z^\ell\mright) \\
     &Z^{\ell+1} &&:= Z^\ell + \stepSize  \left( Y^{\ell+1}-X^{\ell+1} \right) ,
\end{alignedat}
\end{equation}
where $\stepSize \in \mathbb{R}$ is a stepsize parameter. It is known that the matrices $X^\ell$ and $Y^\ell$ converge with rate $\mathcal{O}(1/\ell)$ (in the ergodic sense) to an optimal solution of \eqref{eqn_PPsubReform2} \cite[Thm.~6.5]{he2016convergence} when $\stepSize \in \left(0,\frac{1+\sqrt{5}}{2} \right)$ \cite[Thm.~5.1]{fortin1983chapter}. 

The minimization problems in \eqref{eqn_admmIterates} admit the following closed form solutions, see e.g., \cite[Eq.~3.4]{oliveira2018admm}:
\newcommand{\objMat}{H}
\begin{equation}
\label{eqn_twoProjects}
\begin{aligned}
    &\argmin_{X \succeq 0} L_\rho(X,Y^\ell,Z^\ell)  = \mathcal{P}_{\symmMatSet^{\abs*{\mathcal{B}}}_+}\mleft( Y^\ell + Z^\ell \mright), \\
    &\argmin_{Y \in \linearConSet{\mathcal{B}}} L_\rho(X^{\ell+1}, Y ,Z^\ell)  = \mathcal{P}_{\linearConSet{\mathcal{B}}}\mleft( X^{\ell+1} - \frac{1}{\rho}  \objMat - Z^\ell  \mright),
\end{aligned}
\end{equation}
where $\objMat \in \symmMatSet^{\left| \mathcal{B} \right|}$ is the matrix that is zero everywhere, except for the entry $\objMat_{\emptyset,\emptyset} = 1$. Note that $\langle \objMat, X \rangle = X_{\emptyset,\emptyset}$. The augmented Lagrangian \eqref{eqn_augmentedLagrange} and scheme \eqref{eqn_admmIterates} correspond to the \textit{scaled form} of the ADMM, see e.g., \mbox{\cite[Sect.~3.1.1]{boyd2011distributed}}. Compared to the unscaled form, the scaled form of the ADMM avoids the multiplication of $(1/\rho)$ with $Z^\ell$, see \eqref{eqn_twoProjects}.

We briefly discuss the two projections in \eqref{eqn_twoProjects}. For any $A \in \symmMatSet^{\abs*{\mathcal{B}}}$, 
\begin{align}
\label{eqn_psdProjection}
\mathcal{P}_{\symmMatSet^{\abs*{\mathcal{B}}}_+}\mleft( A \mright) = \sum_{ \lambda \in \spec{A} : \lambda > 0} \lambda u_\lambda u^\top_\lambda,
\end{align}
where $\spec{A}$ denotes the eigenspectrum of $A$, and the vectors $(u_\lambda)_{\lambda \in \spec{A}}$ form an orthonormal basis of eigenvectors. To project a symmetric matrix $Y$ onto $\linearConSet{\mathcal{B}}$, see \eqref{eqn_sdpCons}, we consider $\linearConSet{\mathcal{B}}$ as a set of vectors by identifying $Y$ with its upper triangular entries. To account for the symmetry of $Y$, we replace the terms $Y_{\beta,\beta^\prime} + Y_{\beta^\prime,\beta}$ with $2 Y_{\beta,\beta^\prime}$. From this point of view, it can be seen that $\linearConSet{\mathcal{B}}$ is (up to reordering) a Cartesian product of closed half-spaces, one for each nonempty $\gamma \in \pairwiseUnion{\mathcal{B}} \cap \allStable{G}$. Therefore, projecting onto $\linearConSet{\mathcal{B}}$ is equivalent to projecting onto each half-space separately. Consider such a half-space corresponding to some $\gamma$, defined by $a^\top x \leq -\unitCardFunc{\gamma}$, where $a \in \{1,2\}^p$, for some $p \in \mathbb{N}$. The projection of some $z \in \mathbb{R}^p$ onto this half-space is given by
\begin{align}
\label{eqn_projectProb}
    \argmin_{x \in \mathbb{R}^p : a^\top x \leq -\unitCardFunc{\gamma}} (x-z)^\top \text{Diag}(a) (x-z). 
\end{align}
The presence of $\mathrm{Diag}(a)$ in the objective function ensures that off-diagonal elements of $Y$ are weighted with a factor of 2, as they appear twice in $Y$. The following result shows that \eqref{eqn_projectProb} can be expressed in closed form.
\begin{lemma}
\label{lemma_halfSpaceProject}
Let $p \in \mathbb{N}$, $a ,z \in \mathbb{R}^p$ with $a > 0$, $b \in \mathbb{R}$, and denote by $\mathbf{1}_p \in \mathbb{R}^p$ the all-ones vector. We have that 
\begin{align}
\label{eqn_projectProbResult}
    \argmin_{x \in \mathbb{R}^p : a^\top x \leq b} (x-z)^\top \mathrm{Diag}(a) (x-z) = z - \frac{\max\left\{ a^\top z -b , 0 \right\}
    }{ \mathbf{1}_p^\top a } \mathbf{1}_p.
\end{align}
\end{lemma}
\begin{proof}
Let $f(x) := (x-z)^\top \mathrm{Diag}(a) (x-z)$. We need to determine the minimizer of the problem 
\begin{align}
    \label{eqn_minimizationProb1}
    \min_{x \in \mathbb{R}^p} f(x) \text{ subject to } a^\top x \leq b.
\end{align}
We consider two cases. If $a^\top z \leq b$, then $z$ minimizes \eqref{eqn_minimizationProb1}, since for any $x \in \R^p$, we have $0 = f(z) \leq f(x)$. Here, the inequality follows from the fact that $a > 0$, which makes $\text{Diag}(a)$ positive definite.

If $a^\top z > b$, we consider the Karush-Kuhn-Tucker (KKT) conditions \cite{kuhn1951nonlinear,karush1939minima} corresponding to \eqref{eqn_minimizationProb1}, which state the following: if $(x^*, \lambda^*)$, with $\lambda^* \geq 0$, is a saddle point of the Lagrangian $L(x,\lambda) := f(x) + \lambda\left( a^\top x - b \right)$, then $x^*$ minimizes \eqref{eqn_minimizationProb1}. The saddle point of $L(x,\lambda)$ is computed as follows:
\begin{align*}
    \frac{\partial L(x, \lambda)}{\partial x} = 2 \, \mathrm{Diag}(a) (x-z) + \lambda a = 0 \implies x^* =  z - \frac{\lambda^*}{2} \mathbf{1}_p.
\end{align*}
Solving $\partial L(x^*, \lambda) / \partial \lambda = a^\top x^* -b = a^\top(z- \frac{\lambda^*}{2} \mathbf{1}_p ) - b =0$ for $\lambda^*$ yields $\lambda^* = 2\left(a^\top z - b \right) / ( \mathbf{1}_p^\top a) \geq 0$, where the inequality follows from $a^\top z > b$. Then $x^* = z - \frac{a^\top z-b}{\mathbf{1}^\top a} \mathbf{1}_p$ minimizes \eqref{eqn_minimizationProb1} by the KKT conditions. 

The proof follows by combining the two cases into the form \eqref{eqn_projectProbResult}.
\end{proof}

\subsection{\textcolor{blue}{Upper} bounds on \texorpdfstring{$\alpha(G)$}{a(G)} from the ADMM iterates}
The following lemma shows that any matrix in $\symmMatSet^{\abs*{\mathcal{B}}}_+$ induces an upper bound on $\alpha(G)$. Note that the matrix $X^\ell$ from the ADMM iterates \eqref{eqn_admmIterates} satisfies $X^\ell \in \symmMatSet^{\abs*{\mathcal{B}}}_+$ for all $\ell \in \mathbb{N}$. Thus, any iteration of the ADMM provides an upper bound on $\alpha(G)$.

\newcommand{\psdVarName}{M}
\newcommand{\lBasisName}{\mathcal{B}'}
\begin{lemma}
\label{lemma_posShift}
Let $G = (V,E)$, where $V = [n]$ for some $n \in \mathbb{N}$, $\partialPowerSet{n}{1} \subseteq \mathcal{B} \subseteq \partialPowerSet{n}{n}$, $\psdVarName \in \symmMatSet^{\abs*{\mathcal{B}}}_+$, and let $(f_\beta)_{\beta \in \pairwiseUnion{\mathcal{B}} \cap \allStable{G}}$ satisfy $(\mathbf{x}^{\mathcal{B}})^\top \psdVarName \mathbf{x}^{\mathcal{B}} \equiv \sum_{\beta \in \pairwiseUnion{\mathcal{B}} \cap \allStable{G} } f_\beta x^\beta \mod \edgeIdeal{G}$. We have that the value
\begin{align}   
\label{eqn_lemmaShiftState}
      \lassObj{\psdVarName} := \psdVarName_{\emptyset,\emptyset} + \sum_{\beta \in \pairwiseUnion{\mathcal{B}} \cap \allStable{G} : \beta \neq \emptyset} \max\left\{ f_\beta + \unitCardFunc{\beta},0\right\} 
\end{align}
satisfies $\lassObj{\psdVarName} \geq \alpha(G)$.
\end{lemma}
\begin{proof}

For notational convenience, we define $\lBasisName := \pairwiseUnion{\mathcal{B}} \cap \allStable{G}$. Consider the polynomial
$$
     g(x) :=  (\mathbf{x}^{\mathcal{B}})^\top \psdVarName \mathbf{x}^{\mathcal{B}}  + \sum_{\beta \in \lBasisName : \beta \neq \emptyset} \max\left\{ f_\beta + \unitCardFunc{\beta},0\right\}\mleft(1-x^\beta \mright) + \sum_{\beta \in \lBasisName : \beta \neq \emptyset} \max\left\{ -f_\beta -\unitCardFunc{\beta},0 \right\} x^\beta.
$$
Observe that for any $x \in S_G$, $g(x) \geq 0$ since $\psdVarName \succeq 0$, and $1-x^\beta, x^\beta \geq 0$. Hence, $g \in \posOnG{G}$. 

Let $(g_\beta)_{\beta \in \lBasisName}$ be the coefficients of $g$, i.e., $g \equiv \sum_{\beta \in \lBasisName} g_\beta x^\beta \mod \edgeIdeal{G}$. For nonempty $\beta \in \lBasisName$, we have
\begin{align*}
    g_\beta = f_\beta - \max\left\{f_\beta + \unitCardFunc{\beta},0\right\} + \max\left\{ -f_\beta - \unitCardFunc{\beta},0 \right\} = -\unitCardFunc{\beta}.
\end{align*}
Thus, $g(x) \equiv g_\emptyset - \sum_{i \in [n]} x_i \mod \edgeIdeal{G}$, where $g_\emptyset$ is the constant term of $g$, given by $g_\emptyset = \psdVarName_{\emptyset,\emptyset} +  \sum_{\beta \in \lBasisName : \beta \neq \emptyset} \max\left\{ f_\beta + \unitCardFunc{\beta},0\right\}$. By \eqref{eqn_popReform} and the fact that $g \in \posOnG{G}$, it follows that $g_\emptyset \geq \alpha(G)$. Since $g_\emptyset = \lassObj{\psdVarName}$, this proves the claim.
\end{proof}
Thus, it follows from \Cref{lemma_posShift} that for  $X^\ell$ as in \eqref{eqn_admmIterates} and $v$ as in \eqref{eqn_lemmaShiftState}, $\lassObj{X^\ell}$ is an upper bound on the stability number. 
{
The bounds reported in Section~\ref{section_numericalResults} are
computed, not rigorously verified in the sense of~\cite{Jansson2008}.
In our implementation, the bound $v(X^\ell)$ is computed using floating-point arithmetic.
The eigenvalue decompositions in the ADMM are performed in single precision.
Since Lemma~\ref{lemma_posShift} holds for any $M \in \mathcal{S}_+^{|\mathcal{B}|}$,
the accuracy of the ADMM iterates affects only the quality of the bound, not its validity. Rounding errors are introduced only in the final evaluation of $v(M)$: the computed matrix $M$ may have slightly negative eigenvalues, and the coefficients $f_\beta$ are computed inexactly.   
}

\section{A dynamic basis selection method and ADMM initialization}
\label{section_dynBasisWarmStart}
We provide a method for selecting a basis $\partialPowerSet{n}{1} \subseteq \mathcal{B} \subseteq \partialPowerSet{n}{2}$ for computing $\alpha^{\mathcal{B}}(G)$, see \eqref{eqn_alphaBDef}. We also provide a method for initializing the ADMM iterates $X^1$, $Y^1$, $Z^1$, see \eqref{eqn_admmIterates}. Both these methods are based on an SDP defining the Lov{\'a}sz theta function of a graph $G = ([n],E)$, with $n \in \N$. This SDP is given by:
\begin{equation}
\label{eqn_standardLovasz}
\begin{aligned}
    \vartheta(G) =& \max_{Z \in \symmMatSet^{1+n}_+} \, \,
 &&\sum_{i \in [n]} Z_{\emptyset, i} \\
    &\quad \text{s.t.} &&Z_{i,j} = 0 \text{ for all } \edge{i}{j} \in E \\
    &   && Z_{i,i} = Z_{\emptyset,i} \text{ for all } i \in [n], \, \, Z_{\emptyset,\emptyset} = 1.
\end{aligned}
\end{equation}
Here, the matrix $Z$ is indexed by the $1+n$ elements of $\partialPowerSet{n}{1}$ (see also \cite[Sect.~3.1]{gouveia2010theta}). It can be shown that \eqref{eqn_standardLovasz} is dual to the SDP defining $\alpha^{\partialPowerSet{n}{1}}(G)$, see \eqref{eqn_alphaBCompact}. To solve \eqref{eqn_standardLovasz}, we use the IPM-based SDP solver MOSEK \cite{mosek}. The PSD variable in \eqref{eqn_standardLovasz} is of order $1+n$ and there are $2|E| + 2n+1$ linear equality constraints, which is small enough for IPM-based solvers to be efficient for graphs of sizes considered here. For instance, among the graphs we consider, \texttt{c\_fat200\_5} attains the largest value of $2|E|+2n+1$, with $|E|=11427$ and $n = 200$ (see \Cref{table_gaarTable}). For this graph, MOSEK solves \eqref{eqn_standardLovasz} in approximately 35 seconds.

\subsection{The basis selection method}
\label{section_choosingB}
Given a graph $G$ on $n$ vertices, and a maximum basis size $s \geq 1+n$, our method aims to select a basis $\partialPowerSet{n}{1} \subseteq \mathcal{B} \subseteq \partialPowerSet{n}{2}$, $\abs*{\mathcal{B}} \leq s$
that minimizes $\alpha^{\mathcal{B}}(G)$. The inclusions $\partialPowerSet{n}{1} \subseteq \mathcal{B} \subseteq \partialPowerSet{n}{2}$ can be interpreted in terms of $\linearConSet{\mathcal{B}}$, see \eqref{eqn_sdpCons}, as follows: matrices in $\linearConSet{\partialPowerSet{n}{1}}$ are submatrices of matrices in $\linearConSet{\mathcal{B}}$, which in turn are submatrices of matrices in $\linearConSet{\partialPowerSet{n}{2}}$. 

Recall that the first and second levels of the Lasserre hierarchy for the stable set problem are the SDPs defining $\alpha^{\partialPowerSet{n}{1}}(G)$ and $\alpha^{\partialPowerSet{n}{2}}(G)$ respectively. Thus, our method selects a basis $\mathcal{B}$ such that $\alpha^{\mathcal{B}}(G)$ corresponds to a level of the Lasserre hierarchy intermediate to levels $1$ and $2$, and for some smaller graphs, equal to level $2$. We do not consider bases corresponding to levels $k > 2$, since the value $\alpha^{\partialPowerSet{n}{2}}(G)$, after rounding down to the nearest integer, often closes the gap with $\alpha(G)$, or is intractable to compute.

To explain our method, we define, for $\beta \subseteq [n]$, the binary matrix $Z^\beta := \begin{bmatrix}
    1 \\ \indicatorVec{\beta}
\end{bmatrix} \begin{bmatrix}
    1 \\ \indicatorVec{\beta}
\end{bmatrix}^\top$, indexed by the $1+n$ elements of $\partialPowerSet{n}{1}$. Observe that $Z^\beta$ is feasible for \eqref{eqn_standardLovasz} if and only if $\beta \in \allStable{G}$. Let $\beta \in \allStable{G}$ correspond to a maximum stable set, and consider the binary values $(Z^\beta_{i,j})_{\{i,j\} \notin E}$. Observe that $Z^\beta_{i,j} = 1$ if and only if both vertices $i,j \in \beta$. Thus, the values $1$ in the vector $(Z^\beta_{i,j})_{\{i,j\} \notin E}$ indicate the maximum stable set $\beta$. As such, we would like to include the monomials $x_i x_j$, for which $Z^\beta_{i,j} = 1$, in our basis. In practice however, the matrix $Z^\beta_{i,j}$ is unknown, since a maximum stable set is not known. Therefore, instead of $Z^\beta$, we consider matrix $Z^*$, which denotes an optimal solution of \eqref{eqn_standardLovasz}, and can be considered a semidefinite approximation of $Z^\beta$. Then, we include monomial $x_ix_j$, $\edge{i}{j} \notin E$, in our basis if $Z^*_{i,j}$ is `large enough'.

We now present our basis selection method formally: Compute first $s'$, defined as the cardinality of $\partialPowerSet{n}{2} \setminus \left( \partialPowerSet{n}{1} \cup E \right)$, which is the set of non-edges in $G$. Then we distinguish two cases, based on the given maximum basis size $s$. \\
\textbf{Case 1: $s < 1+n+s'$.}
 Solve \eqref{eqn_standardLovasz} using an IPM-based SDP solver to obtain an optimal solution $Z^*$.  As basis $\mathcal{B}$, we choose the sets in $\partialPowerSet{n}{1}$ and the $s-(1+n)$ non-edges $\{i,j\}$ with largest value $Z^*_{i,j}$. Note that $\abs*{\mathcal{B}} = \left|\partialPowerSet{n}{1}\right|+s-(1+n) = s$. \\
 \textbf{Case 2: $s \geq 1+n+s'$.} We take the basis $\mathcal{B} = \partialPowerSet{n}{2} \setminus E$, of size $\abs*{\mathcal{B}} =1+n+s' \leq s$. Observe that $\mathcal{B} = \partialPowerSet{n}{2} \cap \allStable{G}$, which implies that $\alpha^{\mathcal{B}}(G) = \alpha^{\partialPowerSet{n}{2}}(G)$, see \Cref{lemma_sizeReduction}. Therefore, in Case 2, the resulting basis corresponds to the second level of the Lasserre hierarchy.

{
Preliminary numerical experiments showed that selecting monomials corresponding to the largest entries of $(Z_{i,j}^*)_{\edge{i}{j} \notin E}$, hereafter referred to as \emph{Large}, outperformed the other basis selection approaches. These approaches included random selection, selection based on the smallest entries of $(Z_{i,j}^*)_{\edge{i}{j} \notin E}$, and selection based on the entries closest to the average $\sum_{\edge{i}{j} \in E} Z^*_{i,j}/|E|$. We also tested a basis selection method based on $(Z^*_{\emptyset,i})_{i \in [n]}$, in which all pairwise combinations of the variables corresponding to the largest values were included as monomials.
In the remainder of this section, we refer to the previously described monomial selection methods as  \emph{Random}, \emph{Least}, \emph{Mean}, and \emph{Vertices}, respectively.

In Table~\ref{basis_select}, we present the results obtained using five different basis selection methods on a sample of graphs from the literature. This sample is a subset of the graphs considered in Section~\ref{section_numericalResults}.
The graph \texttt{HoG\_34276} has 72 vertices and 144 edges, 
\texttt{G\_150\_025} has 150 vertices and 2835 edges,
\texttt{Keller4} has 171 vertices and 5100 edges
\texttt{G\_200\_025} has 200 vertices and 4905  edges, and
\texttt{brock200\_1} has 200  vertices and   5066 edges.
We applied five different basis selection methods to those graphs. For each of these 5 bases, we ran the ADMM for 1800 seconds. Table~\ref{basis_select} reports the objective value obtained in the final iteration and the total number of iterations performed by the algorithm.
The results show that, for a graph with fewer vertices, such as  \texttt{HoG\_34276}, the objective values obtained using different basis selection methods differ only slightly.
However, all bounds are close to the stability number of this graph that equals 24.
Note that, for larger graphs, the differences in objective values may be as large as two integer values. The results in the table indicate that the basis selection method used in Section~\ref{section_numericalResults}, referred to as \emph{Large}, shows competitive performance among the alternatives considered.
}
\begin{table}[ht]
\centering {
\begin{tabular}{l*{5}{cc}}
\toprule
\multirow{2}{*}{Graph}
& \multicolumn{2}{c}{Random}
& \multicolumn{2}{c}{Least}
& \multicolumn{2}{c}{Mean}
& \multicolumn{2}{c}{Vertices}
& \multicolumn{2}{c}{Large} \\
\cmidrule(lr){2-3}
\cmidrule(lr){4-5}
\cmidrule(lr){6-7}
\cmidrule(lr){8-9}
\cmidrule(lr){10-11}
& Obj. & Iter.
& Obj. & Iter.
& Obj. & Iter.
& Obj. & Iter.
& Obj. & Iter. \\
\midrule
\texttt{HoG\_34276}
& 24.0515 & 3580
& 24.0187 & 3730 
& 24.0441 & 3800
& 24.0203 & 3670
& 24.0295 & 3772\\
\texttt{G\_150\_025}
& 20.0569 & 3590
& 19.9586 & 3550
& 19.4631 & 3520
& 19.7924 & 3430
& 19.4612 & 3560 \\
\texttt{Keller4 }
& 12.6033 & 3610 
& 12.5835 & 3590 
& 11.6192 & 3620 
& 12.4073 & 3670
& 11.6230 & 3570\\
\texttt{G\_200\_025}
& 24.2855  & 3720
& 24.2375 & 3760
& 25.3420  & 3870
& 24.0549 & 3730
& 23.6781  & 3680 \\
\texttt{brock200\_1}
& 23.4851 & 3840 
& 23.4467 & 3200
& 24.6792  & 3470
& 23.3877 & 3670 
& 22.9186 & 3570\\
\bottomrule
\end{tabular} }
\caption{Basis selection.}
\label{basis_select}
\end{table}

\subsection{Initialization of ADMM iterates}
\label{section_admmInitialize}
Our initialization method is defined for any basis $\partialPowerSet{n}{1} \subseteq \mathcal{B} \subseteq \partialPowerSet{n}{n}$, and depends on optimal primal and dual solutions to \eqref{eqn_standardLovasz}, denoted respectively by $Z^* \in \symmMatSet^{1+n}_+$ and $X^* \in \symmMatSet^{1+n}_+$. 

Note that the initial ADMM iterates $X^1$, $Y^1$, $Z^1$ are indexed by the elements of $\mathcal{B}$. Our initialization method sets the principal submatrix of $X^1$, that is indexed by the elements of $\partialPowerSet{n}{1}$, equal to $X^*$. We set the other elements of $X^1$ to zero, and set $Y^1 = X^1$. We initialize $Z^1$ by setting the principal submatrix of $Z^1$ corresponding to the elements of $\partialPowerSet{n}{1}$ as $(1/\rho) Z^*$, and the rest all zero. Here, we scale $Z^*$ by $1/\rho$, since \eqref{eqn_admmIterates} corresponds to the scaled form of the ADMM.

An important property of this initialization is that $\lassObj{X^1}$, see \eqref{eqn_lemmaShiftState}, equals $\vartheta(G)$. Indeed, since $X^* \succeq 0$, also $X^1 \succeq 0$. Moreover, since $X^*$ is a feasible dual solution to \eqref{eqn_standardLovasz}, and \eqref{eqn_standardLovasz} is the SDP dual of \eqref{eqn_alphaBCompact} for $\mathcal{B} = \partialPowerSet{n}{1}$, it follows that $X^* \in \linearConSet{ \partialPowerSet{n}{1}}$. This implies that the values $f_\beta$ in \eqref{eqn_lemmaShiftState}, corresponding to $X^1$, satisfy $f_\beta \leq -\unitCardFunc{\beta}$, from where it follows that $\lassObj{X^1} = X^1_{\emptyset,\emptyset}$. Lastly, since $X^*$ is an optimal dual solution to \eqref{eqn_standardLovasz}, $X^*_{\emptyset,\emptyset} = X^1_{\emptyset,\emptyset} = \vartheta(G)$. Thus, $\lassObj{X^1} = \vartheta(G)$. Consequently, the best bound returned by the ADMM after a finite number of iterations is at most $\vartheta(G)$.

\section{Computational results}
\label{section_numericalResults}
We compute \LassHierarc{} bounds for the stable set problem, using the ADMM as described in \Cref{section_admmIterates}, with stepsize parameter $\stepSize = 3/2$, see \eqref{eqn_admmIterates}.
We set the ADMM penalty parameter $\rho = (4/5)\sqrt{\abs*{\mathcal{B}}}$, see \eqref{eqn_augmentedLagrange}, where $\mathcal{B}$ is the used basis of size at most $s \in \N$, determined by the basis selection method outlined in \Cref{section_choosingB}. 
{We tuned the parameter $\rho$ by testing different bases on instances of varying sizes $n$. The performance of the method is less sensitive to the choice of $\stepSize$, and any value of $\stepSize \in \left (1,\frac{1+\sqrt{5}}{2}\right )$ performs well.}
The algorithm\footnote{{ Our codes are available upon request.}} 
is implemented in MATLAB.
All experiments are run on a machine with 16GB RAM and an Intel i7-1165G7 CPU.

In \Cref{numerics:duble-single} we compare SDP bounds obtained from two versions of the ADMM algorithm: one that computes the \eigenDecomp{} in \eqref{eqn_psdProjection} using single precision, and another that uses double precision. {We conclude that the single precision ADMM requires less computation time per iteration, without a significant loss in quality of the corresponding bounds.}
The stopping conditions are provided 
in \Cref{section_stoppingCond}, and are used in \Cref{sect:numericBounds} to compute bounds on $\alpha(G)$ from the Lasserre hierarchy  with the single precision ADMM. We compare  these bounds with the bounds obtained by the exact subgraph hierarchy (ESH) \cite{gaar2020computational} and the bounds from \cite{pucher2024class}. Both approaches provide among the strongest SDP bounds for the stable set problem.

\subsection{Single vs.~double precision for \eigenDecomp{}s} \label{numerics:duble-single}
It is well-known that one of the most expensive steps of the ADMM, when applied to SDP, is computing projections onto the PSD cone, see e.g., \cite[Sect.~5]{oliveira2018admm} and \cite[Sect.~1]{rontsis2022efficient}. It is standard to compute these projections $\mathcal{P}_{\symmMatSet^{\abs*{\mathcal{B}}}_+}\mleft( A \mright)$, see \eqref{eqn_psdProjection}, by computing the full \eigenDecomp{} of $A$, that is, $A = \sum_{\lambda \in \spec{A}} \lambda u_\lambda u^\top_\lambda$. Computing the \eigenDecomp{} in single precision is computationally less expensive than in double precision. {However, the lower accuracy of single precision might result in worse upper bounds compared to double precision. We investigate this trade-off by comparing} the SDP bounds on the stable set problem obtained by two versions of the ADMM algorithm: one that uses single precision and another that uses double precision for the \eigenDecomp{}s. All other parts of the algorithms remain the same.

We run each version of the ADMM algorithm for a fixed number of iterations to compute bounds on  $\alpha(G)$  for three different graphs. All ADMM iterates, see \eqref{eqn_admmIterates}, are initialized by setting $X^1=Y^1=Z^1=0$, i.e., the zero matrix of appropriate size. When the iteration number $\ell$ is a multiple of 100, we compute the bound $\lassObj{X^\ell}$, see \eqref{eqn_lemmaShiftState}, which satisfies $\lassObj{X^\ell} \geq \alpha(G)$. We also track the computation time, and present the results in \Cref{table_singleVDouble,table_singleVDouble2,table_singleVDouble3}. In these tables, column `\colNameEigs{}' (for bound difference) reports the double precision bound $\lassObj{X^\ell}$ subtracted from the single precision $\lassObj{X^\ell}$.

From the data presented in these tables we conclude that the difference in $\lassObj{X^\ell}$ between the two precisions and for fixed $\ell$, is negligible (at most 0.04517). In contrast, the reduction in computation time may be significant, see for instance the row corresponding to $\ell = 600$ in \Cref{table_singleVDouble2}: the single precision ADMM requires 113.82 seconds to compute 600 iterations, whereas the double precision ADMM requires 205.35 seconds. With this larger computation time, the double precision bound $\lassObj{X^\ell}$ is only 0.00677 smaller than the single precision bound. Rounded down, both precisions provide the same bound on $\alpha(G)$, but the single precision ADMM requires only 55\% of the computation time of the double precision ADMM. 

We also perform the following similar experiment: we run each ADMM version for one hour on a single graph. We pick a basis of size $s = 2500$  and initialize the ADMM iterates with the methods outlined in \Cref{section_choosingB,section_admmInitialize}, respectively. At the first iteration and every 10 seconds $t$, we compute the upper bound $v_t := \lassObj{X^\ell}$, where $\ell$ is the iteration index at time $t$. We set $v_0 := \lassObj{X^1} = \vartheta(G)$. This yields a set of upper bounds $\{v_0, v_{10}, \dots, v_{3600} \}$. \Cref{fig_singleVDoubleOverTime} reports the best bounds achieved by both ADMM versions over time. That is, \Cref{fig_singleVDoubleOverTime} plots the curves through the points $(t, \min_{k \in \{0,10,\dots,t\}} v_k )$, for $t \in \{0,10,20,\dots,3600\}$. For the single precision version of the ADMM algorithm, the bounds $v_{10}, v_{20}, \dots, v_{80} \geq v_0$, and thus the plot corresponding to single precision is flat for the first 80 seconds. For the double precision variant of the ADMM, the plot is flat for the first 170 seconds. \Cref{fig_singleVDoubleOverTime} demonstrates that the single precision ADMM provides stronger bounds than the double precision ADMM, at any time $t$, $t \leq 3600$. This is due to the fact that the single precision ADMM algorithm can perform more ADMM iterations in the same time as compared to the double precision ADMM algorithm.

Based on these conclusions, we will use the version of the ADMM that computes \eigenDecomp{}s in single precision for the remainder of this section.

\subsection{Warm start vs. cold start for the ADMM}
\label{section_warm_start}

{
In Section~\ref{section_admmInitialize}, we explain how the ADMM iterates are initialized.
In particular, we warm-start the ADMM using an optimal solution of the SDP formulation of the Lov\'asz theta function given in~\eqref{eqn_standardLovasz}. Consequently, after any finite number of iterations, the best bound returned by the ADMM is at most $\vartheta(G)$. 
In the cold-start approach, we initialize the ADMM by setting all variables to zero matrices of appropriate sizes. In this section, we compare the bounds obtained using the warm-start and cold-start approaches.

We compare the upper bounds obtained for the two graphs \texttt{G\_150\_025}  and  $\texttt{sanr200\_0\_9}$.
Graph  $\texttt{G\_150\_025}$ had 150 vertices and 2835 edges, whereas graph  $\texttt{sanr200\_0\_9}$
has 200 vertices and 2037 edges.
For both graphs, we use a basis of size $s=2500$.
The bounds for $\texttt{G\_150\_025}$ (resp.~$\texttt{sanr200\_0\_9}$)  are given in Table~\ref{tab:warm_cold_start_lasserre}  (resp.~Table~\ref{tab:warm_cold_start sanr200}).
The results show that using a warm start is advantageous during the initial iterations, with a more noticeable benefit for the larger graph. As the algorithm progresses, the difference between the bounds decreases, particularly for $\texttt{G\_150\_025}$. The smaller differences observed for this graph may also be due to its stability number, $\alpha(G)=19$.

\begin{table}[ht]
\centering

\begin{tabular}{r|r|r|r}
\toprule 
$\ell$ & Warm-start & Cold-start & Bnd.\ diff. \\[0.5ex]
\hline
1    & 23.718 & 150.000 & 126.282 \\
50   & 22.951 &  29.181 &   6.230 \\
150  & 21.827 &  26.691 &   4.864 \\
250  & 21.121 &  24.273 &   3.151 \\
350  & 20.658 &  22.050 &   1.392 \\
450  & 20.338 &  20.605 &   0.267 \\
550  & 20.107 &  20.051 &   0.056 \\
650  & 19.933 &  19.764 &   0.169 \\
750  & 19.798 &  19.589 &   0.209 \\
850  & 19.690 &  19.471 &   0.219 \\
950  & 19.602 &  19.386 &   0.215 \\
1050 & 19.529 &  19.322 &   0.207 \\
1150 & 19.467 &  19.271 &   0.195 \\
1250 & 19.413 &  19.230 &   0.184 \\
1350 & 19.367 &  19.195 &   0.172 \\
1450 & 19.327 &  19.166 &   0.161 \\
1550 & 19.291 &  19.141 &   0.151 \\
1650 & 19.260 &  19.119 &   0.141 \\
1750 & 19.231 &  19.099 &   0.132 \\
1850 & 19.206 &  19.081 &   0.124 \\
\bottomrule
\end{tabular}
\caption{The warm-start vs.~cold-start  $v(X^\ell)$ for
$\texttt{G\_150\_025}$, $n=150$, $\alpha(G)=19$.}
\label{tab:warm_cold_start_lasserre}
\end{table}

\begin{table}[ht]
\centering 
\begin{tabular}{r|r|r|r}
\toprule
$\ell$ & Warm-start & Cold-start & Bnd.\ diff. \\[0.5ex]
\hline
1    & 49.274 & 200.000 & 150.726 \\
50   & 48.190 &  63.371 &  15.181 \\
150  & 46.521 &  61.310 &  14.789 \\
250  & 45.394 &  58.812 &  13.418 \\
350  & 44.744 &  56.318 &  11.574 \\
450  & 44.380 &  53.838 &   9.459 \\
550  & 44.158 &  51.402 &   7.244 \\
650  & 44.009 &  49.115 &   5.106 \\
750  & 43.902 &  47.257 &   3.355 \\
850  & 43.821 &  45.921 &   2.100 \\
950  & 43.758 &  45.067 &   1.309 \\
1050 & 43.707 &  44.567 &   0.860 \\
1150 & 43.667 &  44.278 &   0.611 \\
\bottomrule
\end{tabular}
\caption{The warm-start vs.~cold-start  $v(X^\ell)$ for
$\texttt{sanr200\_0\_9}$, $n=200$, $\alpha(G)=42$.}
\label{tab:warm_cold_start sanr200}
\end{table}
}

\subsection{Stopping conditions}
\label{section_stoppingCond}
For each graph, we run the ADMM algorithm for at most one hour (unless otherwise specified). Next to maximum running time, we have the following stopping conditions: we stop if
\begin{align}
    \label{eqn_stoppingCondition}
    \max \left\{ \frac{ \left\| X^\ell - Y^\ell \right\|}{ 1+ \left\| X^\ell \right\|}, \rho \frac{ \left\| X^{\ell-1} - X^{\ell}   \right\|}{ 1+ \left\| X^\ell \right\|}\right\} \leq 10^{-4},
\end{align}
for 3 consecutive iterations, see e.g., \cite[Sect.~3.3.1]{boyd2011distributed}. To keep computation costs minimal, we only verify whether  \eqref{eqn_stoppingCondition} holds whenever the iteration number $\ell$ is a multiple of 100 (and then also for the consecutive iterations if \eqref{eqn_stoppingCondition} holds). We also stop earlier if the objective value $X^\ell_{\emptyset, \emptyset}$ stagnates. Specifically, we stop if 
\begin{align}
    \label{eqn_objStagnation}
    \abs*{ X^\ell_{\emptyset,\emptyset} - X^{\ell-1}_{\emptyset,\emptyset} } < 10^{-5}
\end{align}
for $K_\text{stag} \in \mathbb{N}$ (not necessarily consecutive) iterations. For all the tables that follow, except \Cref{table_moreCPU}, we set $K_\mathrm{stag} = 150$.

\newcommand{\tableDescr}[4]{Comparison of upper bounds for \texttt{#1}, $n=#2$, $\alpha(G) = #3$, $s = #4$.}
\begin{table}[ht]
\centering
\begin{tabular}{l|rr|rr|r}
\hline
\multicolumn{1}{c|}{\multirow{2}{*}{$\ell$}} & \multicolumn{2}{c|}{Single precision} & \multicolumn{2}{c|}{Double precision} & \multicolumn{1}{c}{\multirow{2}{*}{\colNameEigs{}}} \\
\multicolumn{1}{c|}{} & $\lassObj{X^\ell}$ & Time (s) & $\lassObj{X^\ell}$ & Time (s) & \multicolumn{1}{c}{} \\ \hline
100 & 7.13476 & 0.14 & 7.13484 & 0.16 & -0.00008 \\
200 & 7.03090 & 0.25 & 7.03090 & 0.32 & 0.00000 \\
300 & 7.07006 & 0.36 & 7.07013 & 0.47 & -0.00007 \\
400 & 7.00445 & 0.46 & 7.00444 & 0.61 & 0.00000 \\
500 & 7.00680 & 0.57 & 7.00677 & 0.76 & 0.00004 \\
600 & 7.00032 & 0.68 & 7.00032 & 0.91 & 0.00001 \\ \hline
\end{tabular}
\caption{\tableDescr{HoG\_15599}{20}{7}{126}}
\label{table_singleVDouble3}
\end{table}

\begin{table}[ht]
\centering
\begin{tabular}{l|rr|rr|r}
\hline
\multicolumn{1}{c|}{\multirow{2}{*}{$\ell$}} & \multicolumn{2}{c|}{Single precision} & \multicolumn{2}{c|}{Double precision} & \multicolumn{1}{c}{\multirow{2}{*}{\colNameEigs{}}} \\
\multicolumn{1}{c|}{} & $\lassObj{X^\ell}$ & Time (s) & $\lassObj{X^\ell}$ & Time (s) & \multicolumn{1}{c}{} \\ \hline
100 & 20.92761 & 5.14 & 20.92756 & 7.99 & 0.00005 \\
200 & 23.01029 & 9.97 & 23.01063 & 16.05 & -0.00035 \\
300 & 17.21289 & 15.28 & 17.21270 & 24.45 & 0.00019 \\
400 & 16.42331 & 21.03 & 16.42139 & 33.03 & 0.00192 \\
500 & 16.51678 & 26.66 & 16.51551 & 41.41 & 0.00127 \\
600 & 16.30446 & 32.15 & 16.30274 & 50.83 & 0.00172 \\
700 & 16.49339 & 37.60 & 16.49267 & 59.29 & 0.00072 \\
800 & 16.29119 & 43.19 & 16.28875 & 67.88 & 0.00245 \\
900 & 16.39126 & 49.01 & 16.38704 & 76.83 & 0.00421 \\
1000 & 16.28313 & 54.91 & 16.27664 & 85.77 & 0.00649 \\
1100 & 16.31098 & 62.01 & 16.30117 & 94.50 & 0.00981 \\ \hline
\end{tabular}
\caption{\tableDescr{MANN\_a9\_clq}{45}{16}{964}}
\label{table_singleVDouble}
\end{table}

\begin{table}[ht]
\centering
\begin{tabular}{l|rr|rr|r}
\hline
\multicolumn{1}{c|}{\multirow{2}{*}{$\ell$}} & \multicolumn{2}{c|}{Single precision} & \multicolumn{2}{c|}{Double precision} & \multicolumn{1}{c}{\multirow{2}{*}{\colNameEigs{}}} \\
\multicolumn{1}{c|}{} & $\lassObj{X^\ell}$ & Time (s) & $\lassObj{X^\ell}$ & Time (s) & \multicolumn{1}{c}{} \\ \hline
100 & 122.53210 & 19.66 & 122.52626 & 35.47 & 0.00584 \\
200 & 82.37548 & 41.58 & 82.37179 & 70.49 & 0.00368 \\
300 & 78.23423 & 59.13 & 78.22998 & 105.84 & 0.00425 \\
400 & 78.89802 & 76.84 & 78.86581 & 139.59 & 0.03221 \\
500 & 71.93651 & 96.10 & 71.89134 & 173.73 & 0.04517 \\
600 & 69.47637 & 113.82 & 69.46960 & 205.35 & 0.00677 \\ \hline
\end{tabular}
\caption{\tableDescr{evil-N150-p98-myc5x30}{150}{60}{1500}}
\label{table_singleVDouble2}
\end{table}

\begin{figure}
    \centering
\begin{tikzpicture}[trim axis left, trim axis right]
\begin{axis}[
    axis lines=left,
    grid=both,
    ylabel style={rotate=-90},
    xlabel={Time (minutes)},
    ylabel style={align=center},
    ylabel={Best upper bound\\on $\alpha(G)$},
    yticklabel style={/pgf/number format/.cd,
    fixed, fixed zerofill,   /pgf/number format/precision=3},
    xmin=0, xmax=60,
    ymin=17,
    ymax=20.5,
    yticklabel style={/pgf/number format/fixed,
                  /pgf/number format/precision=1},
    ytick={17,17.5,18,18.5,19,19.5,20,20.5}, %
    legend pos=north east,
    extra y ticks={17.5,18.02},
    extra y tick style={grid=none,ticks=major,ticklabel pos=right},
    extra y tick labels={$\! 17.46$,$\! 18.02$},
]

\addplot[
    color=black,
    thick
    ]
    coordinates {(0.04,20.23535)(0.19,20.23535)(0.34,20.23535)(0.52,20.23535)(0.68,20.23535)(0.84,20.23535)(1.02,20.23535)(1.18,20.23535)(1.34,20.23535)(1.53,20.23535)(1.68,20.23535)(1.84,20.23535)(2.01,20.23535)(2.17,20.23535)(2.34,20.23535)(2.51,20.23535)(2.68,20.23535)(2.84,20.23535)(3.00,20.13387)(3.18,20.06597)(3.34,20.06597)(3.50,20.06597)(3.69,20.06597)(3.85,20.06597)(4.01,20.06597)(4.19,20.06597)(4.35,20.05977)(4.50,20.00277)(4.67,19.95211)(4.85,19.92292)(5.01,19.91205)(5.17,19.90192)(5.36,19.88753)(5.52,19.87787)(5.68,19.87156)(5.84,19.86341)(6.02,19.84638)(6.18,19.82665)(6.34,19.80677)(6.50,19.79002)(6.69,19.77389)(6.85,19.76161)(7.00,19.75029)(7.19,19.73839)(7.34,19.72909)(7.51,19.71966)(7.69,19.70769)(7.85,19.69668)(8.01,19.68535)(8.17,19.67403)(8.35,19.66113)(8.52,19.65044)(8.67,19.64012)(8.86,19.62844)(9.01,19.61853)(9.17,19.60853)(9.35,19.59671)(9.52,19.58654)(9.68,19.57647)(9.86,19.56499)(10.02,19.55539)(10.17,19.54599)(10.35,19.53519)(10.51,19.52603)(10.67,19.51693)(10.86,19.50635)(11.01,19.49734)(11.18,19.48840)(11.33,19.47956)(11.52,19.46936)(11.68,19.46070)(11.84,19.45209)(12.00,19.44350)(12.18,19.43349)(12.34,19.42488)(12.50,19.41626)(12.68,19.40619)(12.84,19.39757)(13.00,19.38900)(13.18,19.37909)(13.34,19.37069)(13.50,19.36240)(13.69,19.35282)(13.84,19.34470)(14.03,19.33528)(14.18,19.32727)(14.34,19.31929)(14.53,19.31003)(14.69,19.30215)(14.85,19.29432)(15.00,19.28655)(15.19,19.27753)(15.34,19.26985)(15.50,19.26220)(15.69,19.25331)(15.86,19.24570)(16.01,19.23811)(16.17,19.23054)(16.35,19.22174)(16.51,19.21422)(16.67,19.20673)(16.86,19.19806)(17.01,19.19066)(17.17,19.18332)(17.35,19.17480)(17.51,19.16756)(17.67,19.16036)(17.86,19.15201)(18.02,19.14490)(18.18,19.13783)(18.34,19.13080)(18.50,19.12380)(18.69,19.11569)(18.85,19.10877)(19.00,19.10189)(19.18,19.09390)(19.34,19.08708)(19.53,19.07916)(19.69,19.07240)(19.84,19.06566)(20.03,19.05783)(20.18,19.05115)(20.34,19.04451)(20.50,19.03789)(20.69,19.03021)(20.85,19.02366)(21.01,19.01714)(21.19,19.00958)(21.35,19.00314)(21.51,18.99673)(21.67,18.99036)(21.85,18.98295)(22.01,18.97664)(22.17,18.97036)(22.35,18.96307)(22.51,18.95686)(22.67,18.95067)(22.85,18.94348)(23.01,18.93735)(23.17,18.93124)(23.35,18.92416)(23.51,18.91811)(23.67,18.91209)(23.86,18.90511)(24.02,18.89914)(24.18,18.89320)(24.33,18.88729)(24.52,18.88042)(24.69,18.87456)(24.85,18.86873)(25.00,18.86293)(25.19,18.85619)(25.35,18.85046)(25.51,18.84474)(25.67,18.83906)(25.83,18.83340)(26.02,18.82682)(26.18,18.82121)(26.36,18.81470)(26.52,18.80914)(26.68,18.80360)(26.84,18.79810)(27.02,18.79170)(27.18,18.78624)(27.34,18.78080)(27.52,18.77449)(27.68,18.76911)(27.84,18.76376)(28.02,18.75754)(28.19,18.75223)(28.35,18.74695)(28.50,18.74168)(28.69,18.73557)(28.85,18.73035)(29.00,18.72515)(29.19,18.71911)(29.35,18.71397)(29.51,18.70884)(29.67,18.70374)(29.85,18.69780)(30.01,18.69274)(30.17,18.68770)(30.35,18.68185)(30.51,18.67686)(30.67,18.67189)(30.85,18.66612)(31.01,18.66120)(31.19,18.65548)(31.35,18.65060)(31.50,18.64575)(31.69,18.64012)(31.84,18.63531)(32.02,18.62972)(32.18,18.62495)(32.34,18.62020)(32.52,18.61468)(32.68,18.60997)(32.83,18.60528)(33.02,18.59984)(33.18,18.59519)(33.36,18.58980)(33.52,18.58520)(33.67,18.58062)(33.86,18.57530)(34.01,18.57075)(34.17,18.56622)(34.36,18.56097)(34.52,18.55647)(34.68,18.55200)(34.84,18.54754)(35.02,18.54236)(35.18,18.53795)(35.34,18.53355)(35.50,18.52917)(35.69,18.52409)(35.85,18.51975)(36.00,18.51542)(36.19,18.51040)(36.34,18.50612)(36.50,18.50185)(36.69,18.49689)(36.85,18.49266)(37.01,18.48844)(37.17,18.48424)(37.35,18.47936)(37.51,18.47520)(37.69,18.47036)(37.84,18.46624)(38.00,18.46213)(38.19,18.45735)(38.35,18.45327)(38.51,18.44921)(38.69,18.44449)(38.85,18.44046)(39.01,18.43645)(39.17,18.43246)(39.35,18.42781)(39.50,18.42385)(39.69,18.41924)(39.86,18.41532)(40.02,18.41205)(40.17,18.40880)(40.36,18.40490)(40.52,18.40167)(40.68,18.39846)(40.85,18.39461)(41.01,18.39077)(41.19,18.38631)(41.35,18.38251)(41.51,18.37872)(41.68,18.37495)(41.84,18.37119)(42.02,18.36681)(42.18,18.36308)(42.33,18.35936)(42.52,18.35504)(42.68,18.35135)(42.83,18.34768)(43.02,18.34340)(43.18,18.33976)(43.36,18.33552)(43.52,18.33191)(43.68,18.32830)(43.83,18.32471)(44.02,18.32054)(44.17,18.31698)(44.36,18.31285)(44.51,18.30932)(44.67,18.30580)(44.86,18.30172)(45.02,18.29822)(45.17,18.29475)(45.35,18.29071)(45.51,18.28726)(45.67,18.28382)(45.85,18.27983)(46.01,18.27642)(46.19,18.27245)(46.35,18.26906)(46.51,18.26569)(46.67,18.26232)(46.85,18.25841)(47.01,18.25508)(47.17,18.25175)(47.35,18.24789)(47.51,18.24459)(47.67,18.24131)(47.86,18.23749)(48.02,18.23423)(48.18,18.23098)(48.34,18.22774)(48.52,18.22398)(48.68,18.22076)(48.84,18.21756)(49.02,18.21384)(49.18,18.21066)(49.34,18.20748)(49.50,18.20433)(49.68,18.20067)(49.84,18.19754)(50.00,18.19442)(50.19,18.19079)(50.34,18.18769)(50.50,18.18461)(50.68,18.18102)(50.84,18.17796)(51.00,18.17491)(51.18,18.17136)(51.34,18.16833)(51.52,18.16480)(51.68,18.16179)(51.84,18.15879)(52.02,18.15531)(52.18,18.15233)(52.34,18.14937)(52.52,18.14593)(52.68,18.14299)(52.83,18.14005)(53.02,18.13664)(53.18,18.13373)(53.36,18.13034)(53.52,18.12745)(53.67,18.12458)(53.86,18.12123)(54.01,18.11838)(54.17,18.11553)(54.35,18.11222)(54.51,18.10940)(54.67,18.10658)(54.86,18.10331)(55.01,18.10051)(55.18,18.09772)(55.34,18.09495)(55.52,18.09172)(55.68,18.08897)(55.84,18.08622)(56.00,18.08348)(56.19,18.08029)(56.34,18.07757)(56.50,18.07486)(56.69,18.07171)(56.85,18.06902)(57.01,18.06633)(57.19,18.06322)(57.34,18.06055)(57.50,18.05789)(57.68,18.05480)(57.84,18.05217)(58.01,18.04954)(58.19,18.04648)(58.35,18.04387)(58.51,18.04128)(58.69,18.03826)(58.85,18.03569)(59.01,18.03311)(59.19,18.03012)(59.35,18.02757)(59.50,18.02502)(59.69,18.02207)(59.85,18.01954)(60.01,18.01702) };

\addplot[
    color=black,
    thick,
    dotted
    ]
    coordinates {
        (0.04,20.23535)(0.17,20.23535)(0.34,20.23535)(0.51,20.23535)(0.67,20.23535)(0.84,20.23535)(1.01,20.23535)(1.18,20.23535)(1.34,20.23535)(1.51,20.13456)(1.67,20.13456)(1.84,20.13456)(2.00,20.12941)(2.17,20.04333)(2.34,20.01368)(2.51,19.98595)(2.68,19.97397)(2.84,19.93030)(3.00,19.91971)(3.17,19.86185)(3.34,19.85288)(3.51,19.83933)(3.68,19.81005)(3.83,19.78046)(4.00,19.75140)(4.17,19.73306)(4.34,19.71985)(4.51,19.69352)(4.68,19.67222)(4.83,19.66043)(5.00,19.62261)(5.17,19.61109)(5.34,19.57765)(5.51,19.57205)(5.68,19.55693)(5.84,19.52426)(6.01,19.51464)(6.17,19.50358)(6.34,19.46397)(6.51,19.45635)(6.67,19.42349)(6.84,19.41703)(7.00,19.39199)(7.17,19.39199)(7.34,19.36342)(7.50,19.35194)(7.67,19.32678)(7.84,19.32678)(8.00,19.30924)(8.17,19.27734)(8.34,19.26468)(8.51,19.24377)(8.67,19.24114)(8.84,19.21147)(9.00,19.19892)(9.17,19.18719)(9.34,19.16905)(9.51,19.14032)(9.68,19.13306)(9.84,19.11629)(10.01,19.11498)(10.18,19.09508)(10.34,19.05717)(10.51,19.04617)(10.68,19.04390)(10.84,19.03090)(11.01,19.02606)(11.18,19.00003)(11.33,18.98365)(11.51,18.96906)(11.67,18.96389)(11.84,18.95523)(12.00,18.93417)(12.17,18.92098)(12.34,18.92098)(12.51,18.90144)(12.67,18.87848)(12.85,18.87848)(13.01,18.86056)(13.17,18.85434)(13.33,18.83459)(13.50,18.82514)(13.67,18.80797)(13.84,18.80797)(14.00,18.79371)(14.17,18.77361)(14.34,18.77361)(14.51,18.75710)(14.67,18.74766)(14.84,18.72338)(15.01,18.72338)(15.17,18.72177)(15.34,18.69988)(15.50,18.68291)(15.67,18.67839)(15.84,18.67662)(16.00,18.66196)(16.17,18.65475)(16.34,18.63068)(16.51,18.62881)(16.67,18.61735)(16.84,18.60586)(17.00,18.60319)(17.17,18.58230)(17.34,18.57471)(17.50,18.56947)(17.67,18.55206)(17.84,18.55191)(18.01,18.54181)(18.18,18.53194)(18.34,18.52451)(18.50,18.50970)(18.67,18.50820)(18.84,18.48958)(19.01,18.48449)(19.18,18.48253)(19.34,18.46498)(19.51,18.46498)(19.68,18.44213)(19.84,18.44213)(20.01,18.42889)(20.17,18.42666)(20.34,18.42440)(20.50,18.40444)(20.67,18.39603)(20.84,18.39583)(21.00,18.38604)(21.17,18.38604)(21.34,18.36286)(21.51,18.36286)(21.67,18.35296)(21.84,18.34279)(22.01,18.33391)(22.18,18.33368)(22.34,18.31426)(22.51,18.31426)(22.68,18.30176)(22.84,18.30176)(23.01,18.30176)(23.18,18.29204)(23.34,18.27785)(23.51,18.27535)(23.67,18.26687)(23.84,18.25801)(24.01,18.25801)(24.17,18.24275)(24.34,18.23185)(24.51,18.22973)(24.68,18.22483)(24.84,18.20533)(25.01,18.20533)(25.17,18.20533)(25.33,18.20261)(25.50,18.19248)(25.67,18.19248)(25.84,18.18738)(26.00,18.18321)(26.17,18.17230)(26.34,18.17230)(26.50,18.15657)(26.68,18.14739)(26.84,18.14719)(27.00,18.14096)(27.17,18.14096)(27.34,18.12148)(27.50,18.12148)(27.67,18.11973)(27.84,18.11680)(28.00,18.10500)(28.18,18.09859)(28.34,18.08152)(28.51,18.08152)(28.67,18.08152)(28.83,18.07650)(29.01,18.06814)(29.18,18.06814)(29.34,18.05403)(29.51,18.05403)(29.67,18.05365)(29.84,18.05365)(30.00,18.04004)(30.18,18.04004)(30.34,18.04004)(30.51,18.02606)(30.67,18.01970)(30.84,18.01439)(31.00,18.01284)(31.17,18.01284)(31.34,17.99160)(31.50,17.99160)(31.67,17.98529)(31.84,17.98529)(32.00,17.97459)(32.17,17.97459)(32.34,17.97131)(32.51,17.96092)(32.67,17.95497)(32.84,17.95363)(33.00,17.95073)(33.17,17.94332)(33.34,17.94260)(33.51,17.92930)(33.67,17.92930)(33.84,17.92362)(34.00,17.92189)(34.17,17.91901)(34.34,17.91901)(34.50,17.91631)(34.67,17.90253)(34.84,17.89685)(35.00,17.89685)(35.17,17.89557)(35.34,17.89234)(35.50,17.88879)(35.67,17.88114)(35.83,17.87898)(36.01,17.87898)(36.18,17.86274)(36.34,17.86274)(36.51,17.85168)(36.68,17.85168)(36.84,17.84592)(37.01,17.84484)(37.17,17.84484)(37.34,17.83702)(37.51,17.82556)(37.67,17.82556)(37.84,17.82359)(38.00,17.82359)(38.17,17.81786)(38.34,17.81786)(38.50,17.81786)(38.67,17.79533)(38.84,17.79533)(39.01,17.79129)(39.17,17.79129)(39.34,17.79129)(39.51,17.78832)(39.67,17.78832)(39.84,17.78210)(40.01,17.77204)(40.17,17.77204)(40.34,17.76396)(40.51,17.76396)(40.67,17.76396)(40.84,17.75739)(41.00,17.75432)(41.17,17.75432)(41.34,17.74775)(41.50,17.74775)(41.67,17.74775)(41.84,17.73686)(42.01,17.73686)(42.17,17.73686)(42.34,17.72689)(42.50,17.72689)(42.67,17.72689)(42.83,17.72642)(43.01,17.71526)(43.18,17.71526)(43.34,17.71000)(43.50,17.71000)(43.68,17.70365)(43.83,17.70073)(44.01,17.70073)(44.17,17.69250)(44.34,17.68843)(44.51,17.68843)(44.67,17.68318)(44.84,17.68318)(45.01,17.68318)(45.18,17.68318)(45.34,17.67797)(45.51,17.67797)(45.67,17.67654)(45.84,17.66261)(46.00,17.65947)(46.17,17.65528)(46.33,17.65528)(46.50,17.65457)(46.67,17.65277)(46.84,17.65277)(47.01,17.64885)(47.18,17.64096)(47.34,17.63468)(47.51,17.63101)(47.67,17.63101)(47.83,17.63101)(48.01,17.62685)(48.18,17.62448)(48.34,17.62448)(48.51,17.61036)(48.67,17.61036)(48.84,17.61036)(49.01,17.61036)(49.18,17.61036)(49.34,17.61036)(49.50,17.60750)(49.67,17.60017)(49.84,17.60017)(50.01,17.59438)(50.18,17.59438)(50.34,17.59408)(50.51,17.59369)(50.68,17.59369)(50.84,17.57862)(51.00,17.57862)(51.17,17.57862)(51.34,17.57513)(51.51,17.57513)(51.67,17.56803)(51.84,17.56677)(52.00,17.56677)(52.18,17.56677)(52.34,17.55576)(52.50,17.55576)(52.67,17.55539)(52.84,17.55131)(53.01,17.55131)(53.17,17.54426)(53.34,17.54426)(53.51,17.54426)(53.67,17.54426)(53.84,17.54052)(54.01,17.53186)(54.17,17.53065)(54.34,17.53065)(54.51,17.52439)(54.67,17.52439)(54.84,17.52439)(55.01,17.52098)(55.17,17.51929)(55.34,17.51762)(55.50,17.51502)(55.67,17.51502)(55.84,17.51502)(56.01,17.51502)(56.17,17.51230)(56.34,17.51230)(56.51,17.50903)(56.67,17.50903)(56.84,17.50838)(57.00,17.50522)(57.18,17.49819)(57.34,17.49626)(57.51,17.49626)(57.68,17.49139)(57.84,17.48775)(58.01,17.48775)(58.17,17.48775)(58.34,17.48344)(58.50,17.48344)(58.67,17.47802)(58.84,17.47802)(59.01,17.47802)(59.17,17.47802)(59.34,17.47202)(59.51,17.47077)(59.68,17.46883)(59.84,17.46883)(60.01,17.46174)
    };

    \legend{Double precision ADMM,Single precision ADMM}
\end{axis}
\end{tikzpicture}
\caption{\tableDescr{evil-N184-p98-myc23x8}{184}{16}{2500}} 
\label{fig_singleVDoubleOverTime}
\end{figure}

\subsection{\lassColName{} bounds on \texorpdfstring{$\alpha(G)$}{alpha(G)}} \label{sect:numericBounds}
We investigate the quality of the upper bounds on the stability number of graphs obtained from the Lasserre hierarchy, and computed via the ADMM. For the remainder of this section, we refer to such bounds as \lassColName{} bounds. For all the tables that follow, except \Cref{table_moreCPU}, we set the maximum basis size $s =2500$. 
This value is chosen to ensure the ADMM converges within one hour on most graphs.
In the remainder of this section, we initialize the ADMM iterates using the method described in \Cref{section_admmInitialize}.

\subsubsection{SDP bounds for evil, random and near-regular graphs}
\label{section_evailRandNearRegRand}
We benchmark the \lassColName{} bounds on the complement of evil graphs\footnote{The evil graphs are available at \href{https://github.com/zbogdan/evil2/tree/main/evil-tests}{\texttt{https://github.com/zbogdan/evil2/tree/main/evil-tests}}.} \cite{szabo2019benchmark}, as well as on random graphs and near-regular graphs. These graphs are also tested in \cite{pucher2024class}, and can be described as follows:
\begin{itemize}
    \item \textbf{Evil graphs \cite{szabo2019benchmark}.} These are benchmark graphs for the \NpComplexity{}-hard clique problem. The clique problem on a graph $G$ is equivalent to to the stable set problem on the complement graph of $G$. Therefore, we consider the complement of evil graphs. 
    \item \textbf{Random graphs.} These are generated following the Erd\"os-R\'enyi model. That is, given $n \in \mathbb{N}$ and $p \in (0,1)$, generate a graph by taking $n$ vertices and creating edges $\edge{i}{j}$ independently with probability $p$ for every $i,j \in [n]$.
    \item \textbf{Near-regular graphs.} For given $n,r \in \mathbb{N}$ such that $nr$ is even, these graphs are constructed as follows (see also \cite[Sect.~7.2]{gaar2020computational}): consider a set of $nr$ vertices given by $\widetilde{V}:= \setFunct{ \{i,b\} }{i \in [n], b \in [r]}$. Select a perfect matching on the vertices in $\widetilde{V}$ to obtain the edge set $\widetilde{E} \subseteq \widetilde{V} \times \widetilde{V}$. Consider now the graph $G$ with vertices $V = [n]$ and edge set $E = \setFunct{ \{i,j\} }{ \exists b,b' \in [r] \text{ s.t. } \{ \{i,b\}, \{j,b'\} \} \in \widetilde{E}}$. Note that $G$ is a regular multigraph. Remove from $G$ any parallel edges and self-loops. The resulting graph is said to be near-regular.
\end{itemize}
We use the  same exact random and near-regular graphs as in \cite{pucher2024class}. For each graph, we use the method provided in \Cref{section_choosingB} to select a basis of size 2500 for computing the \lassColName{} bounds. Because $2500 < \abs*{\partialPowerSet{n}{2} \cap \allStable{G}}$ for all graphs, the resulting bounds correspond to the Lasserre hierarchy at levels intermediate to 1 and 2.

\Cref{table_jan15Results} reports the \lassColName{} bounds on $\alpha(G)$, computed via the ADMM, in the column `\lassColName{}'. Columns $n$, $\abs*{E}$, and $\vartheta(G)$ report the number of vertices, number of edges, and Lov\'asz theta function respectively, for each graph. Column $\alpha(G)$ reports (bounds on) the stability number of the graphs. For evil graphs, the value of $\alpha(G)$ is known by construction. The values and intervals for the stability numbers of the random and near-regular graphs are taken from \cite[Table 6]{gaar2022sdp}. In \Cref{table_moreCPU}, we present improved bounds on $\alpha(G)$ compared to \cite[Table 6]{gaar2022sdp}.

The columns under `BOUND~2 \cite{pucher2024class}' in \Cref{table_jan15Results} report BOUND~2 from the recent paper \cite[Sect.~4.1]{pucher2024class}, which provides { the strongest SDP-based upper bounds on $\alpha(G)$ previously reported in the literature for the listed graphs.} BOUND~2 is obtained by strengthening the Lov{\'a}sz theta function with additional valid inequalities, such as triangle inequalities and inequalities induced by complete subgraphs of $G$ (i.e., constraints of the form $\sum_{i \in U} x_i \leq 1$ where $U$ is a set of pairwise connected vertices). BOUND~2 is computed by the IPM-based SDP solver MOSEK \cite{mosek}. We computed BOUND 2 on the same machine  we used to compute the \lassColName{} bounds\footnote{The code for BOUND 2 is available at \href{https://arxiv.org/src/2401.17069v2/anc}{\texttt{https://arxiv.org/src/2401.17069v2/anc}}.}. Additional details regarding BOUND~2 can be found in \cite[Sect.~4.1]{pucher2024class}. 

The columns `\colName{}' (for gap closed) {under \lassColName{} and BOUND 2} report the fraction $\frac{\vartheta(G) - f^*}{\vartheta(G) - \alpha(G)}$ {rounded down to the first digit, where $f^*$ equals the value of the corresponding bound.} Note that $f^* \in [\alpha(G),\vartheta(G)]$. For the three graphs with unknown $\alpha(G)$, we use the best-known lower bound on $\alpha(G)$ from \cite[Table 6]{gaar2022sdp} to compute \colName{}. Lastly, bounds that equal $\alpha(G)$ when rounded down are boldfaced.

{ The computational results in \Cref{table_jan15Results} show that neither BOUND~2 nor our bound dominates the other.} Computing BOUND~2 for the graphs in \Cref{table_jan15Results} requires on average only 4 minutes of computation time per graph. The ADMM required on average 39 minutes per graph. 
{The additional computational effort sometimes pays off. In particular, the \lassColName{} bounds are stronger than BOUND~2 for various graphs, sometimes closing the gap to $\alpha(G)$ when BOUND~2 does not} (see the graphs \texttt{evil-N125-p98-s3m25x5}, \texttt{evil-N138-p98-myc23x6}, \texttt{reg\_n100\_r6} and \texttt{reg\_n100\_r8}). 
There is one graph in \Cref{table_jan15Results}, \texttt{rand\_n200\_p002}, for which the \lassColName{} bound does not improve upon the Lov{\'a}sz theta function $\vartheta(G) = \alpha^{\partialPowerSet{n}{1}}(G) =95.778$. We ran our ADMM algorithm for 4 hours to recompute the \lassColName{} bound for this graph, which resulted in an improved bound of $95.244$, see also \Cref{table_improvedLassSDPName}.
{
One explanation for the difficulty in improving this bound is that rounding down the Lov\'asz theta number already yields the optimal value.
}

We rerun the ADMM algorithm on the graphs \texttt{reg\_n200\_r10} and \texttt{reg\_n200\_r6}  with increased basis size $\abs*{\mathcal{B}}$ and extended running time.   These changes yield improved upper bounds on $\alpha(G)$ compared to those in \cite[Table 6]{gaar2022sdp}, see also \Cref{table_jan15Results}. The improved bounds are reported in \Cref{table_moreCPU}, where column `$\abs*{\mathcal{B}}$' reports the used basis size and column `T. (min.)' reports the runtime rounded to the nearest minute. In these computations, we use $K_\text{stag} = 500$, see \eqref{eqn_objStagnation}.

Lastly, we compare the \lassColName{} bounds also with the ESH approach from \cite{gaar2020computational}. Specifically, we compare the \lassColName{} bounds with the ESH bounds on the random and near-regular graphs reported in Tables 4 and 5 of \cite{gaar2020computational}. Computational results are presented in \Cref{table_newComparison}, with identical columns as \Cref{table_jan15Results}, and the newly added column `ESH'. The columns under `ESH' report the obtained bounds by the ESH and corresponding \colName{} values. The \lassColName{} bounds improve over the ESH bounds in 11 of the 15 reported graphs. 
The computation times of these ESH bounds is reported in \cite{gaar2020computational}, although the used computer is not specified. For each of the graphs in \Cref{table_newComparison}, the ESH required on average 1454 seconds. Note that \lassColName{} bounds, when rounded down to the closest integer, close the gap for more graphs than the other two approaches. 
\begin{table}[ht]
\centering
\begin{tabular}{lrrr|rr|rr|r}
\hline
\multicolumn{1}{l}{\multirow{2}{*}{Graph   name}} & \multicolumn{1}{c}{\multirow{2}{*}{$n$}} & \multicolumn{1}{c}{\multirow{2}{*}{$\abs*{E}$}} & \multicolumn{1}{c|}{\multirow{2}{*}{$\alpha(G)$}} & \multicolumn{2}{c|}{\lassColName{}} & \multicolumn{2}{c|}{BOUND 2   \cite{pucher2024class}} & \multicolumn{1}{c}{\multirow{2}{*}{$\vartheta(G)$}} \\
\multicolumn{1}{c}{} & \multicolumn{1}{c}{} & \multicolumn{1}{c}{} & \multicolumn{1}{c|}{} & \multicolumn{1}{l}{Bound} & \multicolumn{1}{l|}{\colName{}} & \multicolumn{1}{l}{Bound} & \multicolumn{1}{l|}{\colName{}} & \multicolumn{1}{c}{} \\ \hline
\texttt{evil-N120-p98-chv12x10} & 120 & 545 & 20 & \textbf{20.355} & 92.1\% & \textbf{20.000} & 99.9\% & 24.526 \\
\texttt{evil-N120-p98-myc5x24} & 120 & 236 & 48 & \textbf{48.466} & 89.8\% & \textbf{48.000} & 99.9\% & 52.607 \\
\texttt{evil-N121-p98-myc11x11} & 121 & 508 & 22 & \textbf{22.361} & 91.7\% & \textbf{22.000} & 99.9\% & 26.397 \\
\texttt{evil-N125-p98-s3m25x5} & 125 & 873 & 20 & \textbf{20.291} & 94.1\% & 22.361 & 52.7\% & 25.000 \\
\texttt{evil-N138-p98-myc23x6} & 138 & 1242 & 12 & \textbf{12.501} & 84.2\% & 15.177 & 0.0\% & 15.177 \\
\texttt{evil-N150-p98-myc5x30} & 150 & 338 & 60 & 61.477 & 71.1\% & \textbf{60.000} & 99.9\% & 65.121 \\
\texttt{evil-N150-p98-s3m25x6} & 150 & 1102 & 24 & 25.239 & 79.3\% & 26.833 & 52.7\% & 30.000 \\
\texttt{evil-N154-p98-myc11x14} & 154 & 701 & 28 & \textbf{28.316} & 94.3\% & \textbf{28.000} & 99.9\% & 33.596 \\
\texttt{evil-N180-p98-chv12x15} & 180 & 944 & 30 & \textbf{30.391} & 94.2\% & \textbf{30.000} & 99.9\% & 36.788 \\
\texttt{evil-N180-p98-myc5x36} & 180 & 439 & 72 & 75.626 & 29.0\% & \textbf{72.000} & 99.9\% & 77.110 \\
\texttt{evil-N184-p98-myc23x8} & 184 & 1764 & 16 & 17.504 & 64.4\% & 20.235 & 0.0\% & 20.235 \\
\texttt{evil-N187-p98-myc11x17} & 187 & 901 & 34 & \textbf{34.332} & 95.1\% & \textbf{34.000} & 99.9\% & 40.795 \\
\texttt{evil-N200-p98-s3m25x8} & 200 & 1550 & 32 & 35.979 & 50.2\% & 35.777 & 52.7\% & 40.000 \\
\texttt{evil-N210-p98-myc5x42} & 210 & 541 & 84 & 86.513 & 58.1\% & \textbf{84.000} & 99.9\% & 90.001 \\
\texttt{evil-N220-p98-myc11x20} & 220 & 1130 & 40 & 42.791 & 65.0\% & \textbf{40.000} & 99.9\% & 47.994 \\
\texttt{evil-N230-p98-myc23x10} & 230 & 2263 & 20 & 22.553 & 51.7\% & 25.294 & 0.0\% & 25.294 \\
\texttt{evil-N240-p98-chv12x20} & 240 & 1352 & 40 & \textbf{40.461} & 94.9\% & \textbf{40.000} & 99.9\% & 49.051 \\
\texttt{evil-N240-p98-myc5x48} & 240 & 718 & 96 & 98.627 & 44.0\% & \textbf{96.011} & 99.7\% & 100.696 \\
\texttt{evil-N250-p98-s3m25x10} & 250 & 2050 & 40 & 45.003 & 49.9\% & 44.721 & 52.7\% & 50.000 \\
\texttt{evil-N253-p98-myc11x23} & 253 & 1456 & 46 & 49.447 & 62.5\% & \textbf{46.014} & 99.8\% & 55.193 \\
\texttt{evil-N300-p98-myc5x60} & 300 & 1033 & 120 & 121.876 & 34.4\% & \textbf{120.020} & 99.2\% & 122.861 \\
\texttt{rand\_n100\_p004} & 100 & 212 & 45 & \textbf{45.140} & 86.8\% & \textbf{45.027} & 97.4\% & 46.067 \\
\texttt{rand\_n100\_p006} & 100 & 303 & 38 & \textbf{38.199} & 91.5\% & \textbf{38.435} & 81.5\% & 40.361 \\
\texttt{rand\_n100\_p008} & 100 & 443 & 32 & \textbf{32.475} & 83.3\% & \textbf{32.433} & 84.7\% & 34.847 \\
\texttt{rand\_n100\_p010} & 100 & 489 & 32 & \textbf{32.029} & 98.5\% & \textbf{32.151} & 92.5\% & 34.020 \\
\texttt{rand\_n200\_p002} & 200 & 407 & 95 & \textbf{95.778} & 0.0\% & \textbf{95.043} & 94.5\% & \textbf{95.778} \\
\texttt{rand\_n200\_p003} & 200 & 631 & 81 & 82.425 & 46.4\% & \textbf{81.079} & 97.0\% & 83.662 \\
\texttt{rand\_n200\_p004} & 200 & 816 & 67 & 69.890 & 58.1\% & 69.818 & 59.2\% & 73.908 \\
\texttt{rand\_n200\_p005} & 200 & 991 & 64 & 67.355 & 33.4\% & 65.544 & 69.3\% & 69.039 \\
\texttt{reg\_n100\_r10} & 100 & 474 & 28 & 29.636 & 56.9\% & 29.431 & 62.3\% & 31.797 \\
\texttt{reg\_n100\_r4} & 100 & 195 & 40 & \textbf{40.333} & 90.3\% & \textbf{40.713} & 79.3\% & 43.449 \\
\texttt{reg\_n100\_r6} & 100 & 294 & 34 & \textbf{34.667} & 82.5\% & 35.047 & 72.5\% & 37.815 \\
\texttt{reg\_n100\_r8} & 100 & 377 & 31 & \textbf{31.645} & 81.4\% & 32.063 & 69.4\% & 34.480 \\
\texttt{reg\_n200\_r10} & 200 & 980 & $[57,59]$ & 60.052 & 67.5\% & 62.695 & 39.5\% & 66.418 \\
\texttt{reg\_n200\_r4} & 200 & 400 & 81 & 82.450 & 78.5\% & 82.246 & 81.5\% & 87.759 \\
\texttt{reg\_n200\_r6} & 200 & 593 & $[69,72]$ & 72.685 & 64.1\% & 73.709 & 54.1\% & 79.276 \\
\texttt{reg\_n200\_r8} & 200 & 792 & $[60,63]$ & 64.749 & 55.9\% & 66.789 & 37.0\% & 70.790 \\ \hline
\end{tabular}
\caption{Results on evil, random and near-regular graphs.}
\label{table_jan15Results}
\end{table}

\begin{table}[ht]
\centering
\begin{tabular}{lrrrr|crr}
\hline
\multicolumn{1}{l}{\multirow{2}{*}{Graph   name}} & \multicolumn{1}{c}{\multirow{2}{*}{$n$}} & \multicolumn{1}{c}{\multirow{2}{*}{$\left|E\right|$}} & \multicolumn{2}{c|}{$\alpha(G)$ bounds} & \multicolumn{1}{c}{\multirow{2}{*}{\lassColName{}}} & \multicolumn{1}{c}{\multirow{2}{*}{$\abs*{\mathcal{B}}$}} & \multicolumn{1}{c}{\multirow{2}{*}{T. (min.)}} \\
\multicolumn{1}{c}{}                              & \multicolumn{1}{c}{}                     & \multicolumn{1}{c}{}                          & New          & Old \cite{gaar2022sdp}   & \multicolumn{1}{c}{}                           & \multicolumn{1}{c}{}                     & \multicolumn{1}{c}{}                           \\ \hline
\texttt{reg\_n200\_r10}                           & 200                                      & 980                                           & $[57,58]$    & $[57,59]$                & 58.932                                         & 4250                                     & 180                                            \\
\texttt{reg\_n200\_r6}                            & 200                                      & 593                                           & $[69,71]$    & $[69,72]$                & 71.821                                         & 3750                                     & 105                                            \\ \hline
\end{tabular}
\caption{Improved bounds on $\alpha(G)$ with larger basis size $\abs*{\mathcal{B}}$ and longer running time.}
\label{table_moreCPU}
\end{table}

\begin{table}[]
\begin{tabular}{lr|rr|rr|rr|r}
\hline
\multicolumn{1}{l}{\multirow{2}{*}{Graph   name}} & \multicolumn{1}{c|}{\multirow{2}{*}{$\alpha(G)$}} & \multicolumn{2}{c|}{\lassColName{}} & \multicolumn{2}{c|}{BOUND 2   \cite{pucher2024class}} & \multicolumn{2}{c|}{ESH   \cite{gaar2020computational}} & \multicolumn{1}{c}{\multirow{2}{*}{$\vartheta(G)$}} \\
\multicolumn{1}{c}{} & \multicolumn{1}{c|}{} & \multicolumn{1}{l}{Bound} & \multicolumn{1}{l|}{\colName{}} & \multicolumn{1}{l}{Bound} & \multicolumn{1}{l|}{\colName{}} & \multicolumn{1}{l}{Bound} & \multicolumn{1}{l|}{\colName{}} & \multicolumn{1}{c}{} \\ \hline
\texttt{rand\_n100\_p004} & 45 & \textbf{45.140} & 86.8\% & \textbf{45.027} & 97.4\% & \textbf{45.021} & 98.0\% & 46.067 \\
\texttt{rand\_n100\_p006} & 38 & \textbf{38.199} & 91.5\% & \textbf{38.435} & 81.5\% & \textbf{38.439} & 81.4\% & 40.361 \\
\texttt{rand\_n100\_p008} & 32 & \textbf{32.475} & 83.3\% & \textbf{32.433} & 84.7\% & \textbf{32.579} & 79.6\% & 34.847 \\
\texttt{rand\_n100\_p010} & 32 & \textbf{32.029} & 98.5\% & \textbf{32.151} & 92.5\% & \textbf{32.191} & 90.5\% & 34.020 \\
\texttt{rand\_n200\_p002} & 95 & \textbf{95.778} & 0.0\% & \textbf{95.043} & 94.5\% & \textbf{95.032} & 95.8\% & \textbf{95.778} \\
\texttt{rand\_n200\_p003} & 81 & 82.425 & 46.4\% & \textbf{81.079} & 97.0\% & \textbf{81.224} & 91.5\% & 83.662 \\
\texttt{rand\_n200\_p004} & 67 & 69.890 & 58.1\% & 69.818 & 59.2\% & 70.839 & 44.4\% & 73.908 \\
\texttt{rand\_n200\_p005} & 64 & 67.355 & 33.4\% & 65.544 & 69.3\% & 66.091 & 58.5\% & 69.039 \\
\texttt{reg\_n100\_r4} & 40 & \textbf{40.333} & 90.3\% & \textbf{40.713} & 79.3\% & \textbf{40.687} & 80.0\% & 43.449 \\
\texttt{reg\_n100\_r6} & 34 & \textbf{34.667} & 82.5\% & 35.047 & 72.5\% & 35.246 & 67.3\% & 37.815 \\
\texttt{reg\_n100\_r8} & 31 & \textbf{31.645} & 81.4\% & 32.063 & 69.4\% & 32.190 & 65.8\% & 34.480 \\
\texttt{reg\_n200\_r10} & $[57,59]$ & 60.052 & 67.5\% & 62.695 & 39.5\% & 62.894 & 37.4\% & 66.418 \\
\texttt{reg\_n200\_r4} & 81 & 82.450 & 78.5\% & 82.246 & 81.5\% & 83.732 & 59.5\% & 87.759 \\
\texttt{reg\_n200\_r6} & $[69,72]$ & 72.685 & 64.1\% & 73.709 & 54.1\% & 75.555 & 36.2\% & 79.276 \\
\texttt{reg\_n200\_r8} & $[60,63]$ & 64.749 & 55.9\% & 66.789 & 37.0\% & 67.785 & 27.8\% & 70.790 \\ \hline
\end{tabular}
\caption{Comparison of the \lassColName{}, BOUND 2 and ESH bounds.}
\label{table_newComparison}
\end{table}

\subsubsection{SDP bounds on graphs from \texorpdfstring{\cite{gaar2024different}}{Gaar}}
We investigate the quality of the \lassColName{} bounds on the graphs tested in \cite{gaar2024different}. This set of graphs\footnote{The graphs used in \cite{gaar2024different} are available at \href{https://arxiv.org/src/2003.13605v6/anc/Data_InputGraphs.mat}{\texttt{https://arxiv.org/src/2003.13605v6/anc/Data\_InputGraphs.mat}}.} contains, among others, complements of DIMACS graphs, additional random Erd\"os-R\'enyi graphs, a spin glass graph, graphs from \cite{brinkmann2013house}, as well as a Paley, a circulant and a cubic graph. Note that stability numbers of the DIMACS graphs are known, see e.g., \cite{balas1996weighted}. 

For all graphs, we use the method from \Cref{section_choosingB} to select a basis of size at most 2500 for the ADMM. The tested graphs on $n \leq 80$ vertices satisfy $\abs*{\partialPowerSet{n}{2} \cap \allStable{G}} \leq 2500$, which implies that we compute the full second level of the Lasserre hierarchy for those graphs.
 The results are reported in \Cref{table_gaarTable}, with the same  column definitions as in~\Cref{table_jan15Results}, except for the new column denoted by $|\mathcal{B}|$. This column reports the size of the used basis (in \Cref{table_jan15Results}, $|\mathcal{B}| = 2500$ for every graph). Bounds that equal $\alpha(G)$ when rounded down are boldfaced.
 
 We again compare our \lassColName{} bounds with BOUND~2 from \cite{pucher2024class}. Both these bounds are stronger than the bounds in \cite{gaar2024different}. Computing BOUND~2 for the graphs in \Cref{table_gaarTable} requires on average only two minutes of computation time per graph. However, for the larger graphs ($n \geq 150$), BOUND~2 cannot be computed due to insufficient memory, as indicated by the table entry `$-$'. The large memory requirement of BOUND~2 is due to the large number of linear constraints.
  {The \lassColName{} bounds improve significantly over $\vartheta(G)$ and, for several instances, also improve upon BOUND~2.}
  In particular, the \lassColName{} bounds often close the gap towards $\alpha(G)$ when rounded down.
 {However, the most significant advantage of our approach is that it produces strong bounds for graph instances for which BOUND~2 cannot be computed. }
 
 All graphs with $n < 72$ took at most 122 seconds, except for \texttt{G\_60\_025} which required approximately 270 seconds. For graphs with $n \geq 72$, the ADMM algorithm required between 15 and 60 minutes minutes to terminate. 

 For the graph \texttt{c\_fat200\_5} in \Cref{table_gaarTable}, the \lassColName{} bound does not improve over $\vartheta(G) = 60.345$. We ran our ADMM algorithm for 4 hours to recompute the \lassColName{} bound of this graph, which resulted in an improved bound of $60.317$, see also \Cref{table_improvedLassSDPName}.
 { The graph $\texttt{c\_fat200\_5}$ is highly symmetric and has a large automorphism group.  For highly symmetric graphs, an appropriate basis selection may take the automorphism group of the graph into account. One possible approach is to select basis elements from several distinct orbitals.}

\begin{table}[ht]
\centering
\begin{tabular}{lrrr|rrr|rr|r}
\hline
\multicolumn{1}{l}{\multirow{2}{*}{Graph   name}} & \multicolumn{1}{c}{\multirow{2}{*}{$n$}} & \multicolumn{1}{c}{\multirow{2}{*}{$\abs*{E}$}} & \multicolumn{1}{c|}{\multirow{2}{*}{$\alpha(G)$}} & \multicolumn{3}{c|}{\lassColName{}} & \multicolumn{2}{c|}{BOUND 2   \cite{pucher2024class}} & \multicolumn{1}{c}{\multirow{2}{*}{$\vartheta(G)$}} \\
\multicolumn{1}{c}{} & \multicolumn{1}{c}{} & \multicolumn{1}{c}{} & \multicolumn{1}{c|}{} & \multicolumn{1}{l}{$\abs*{\mathcal{B}}$} & \multicolumn{1}{l}{Bound} & \multicolumn{1}{l|}{\colName{}} & \multicolumn{1}{l}{Bound} & \multicolumn{1}{l|}{\colName{}} & \multicolumn{1}{c}{} \\ \hline
\texttt{HoG\_34272} & 9 & 17 & 3 & 29 & \textbf{3.000} & 99.8\% & \textbf{3.000} & 99.9\% & \textbf{3.338} \\
\texttt{HoG\_15599} & 20 & 44 & 7 & 167 & \textbf{7.003} & 99.6\% & \textbf{7.000} & 99.9\% & \textbf{7.820} \\
\texttt{CubicVT26\_5} & 26 & 39 & 10 & 313 & \textbf{10.100} & 94.5\% & \textbf{10.662} & 63.5\% & 11.817 \\
\texttt{HoG\_34274} & 36 & 72 & 12 & 595 & \textbf{12.010} & 99.1\% & \textbf{12.000} & 99.9\% & 13.232 \\
\texttt{HoG\_6575} & 45 & 225 & 10 & 811 & \textbf{10.132} & 97.3\% & 13.220 & 36.2\% & 15.053 \\
\texttt{MANN\_a9\_clq} & 45 & 72 & 16 & 964 & \textbf{16.281} & 80.9\% & 17.225 & 16.9\% & 17.475 \\
\texttt{Circulant47\_030} & 47 & 282 & 13 & 847 & \textbf{13.003} & 99.7\% & \textbf{13.026} & 97.9\% & 14.302 \\
\texttt{G\_50\_025} & 50 & 308 & 12 & 968 & \textbf{12.028} & 98.2\% & \textbf{12.367} & 76.5\% & 13.564 \\
\texttt{G\_60\_025} & 60 & 450 & 13 & 1381 & \textbf{13.003} & 99.7\% & \textbf{13.241} & 81.1\% & 14.281 \\
\texttt{PaleyGraph61} & 61 & 915 & 5 & 977 & \textbf{5.289} & 89.7\% & 7.810 & 0.0\% & 7.810 \\
\texttt{hamming6\_4} & 64 & 1312 & 4 & 769 & \textbf{4.032} & 97.5\% & \textbf{4.749} & 43.7\% & 5.333 \\
\texttt{HoG\_34276} & 72 & 144 & 24 & 2485 & \textbf{24.042} & 98.2\% & \textbf{24.000} & 99.9\% & 26.463 \\
\texttt{G\_80\_050} & 80 & 1620 & 9 & 1621 & \textbf{9.001} & 99.7\% & \textbf{9.092} & 78.8\% & \textbf{9.435} \\
\texttt{G\_100\_025} & 100 & 1243 & 17 & 2500 & \textbf{17.000} & 99.9\% & 18.428 & 41.5\% & 19.441 \\
\texttt{spin5} & 125 & 375 & 50 & 2500 & \textbf{50.236} & 95.9\% & \textbf{50.000} & 99.9\% & 55.902 \\
\texttt{G\_150\_025} & 150 & 2835 & 19 & 2500 & \textbf{19.127} & 97.3\% & \multicolumn{2}{c|}{$-$} & 23.718 \\
\texttt{keller4} & 171 & 5100 & 11 & 2500 & \textbf{11.622} & 79.3\% & \multicolumn{2}{c|}{$-$} & 14.012 \\
\texttt{G\_200\_025} & 200 & 4905 & 21 & 2500 & 23.656 & 63.1\% & \multicolumn{2}{c|}{$-$} & 28.217 \\
\texttt{brock200\_1} & 200 & 5066 & 21 & 2500 & 22.912 & 70.3\% & \multicolumn{2}{c|}{$-$} & 27.457 \\
\texttt{c\_fat200\_5} & 200 & 11427 & 58 & 2500 & 60.345 & 0.0\% & \multicolumn{2}{c|}{$-$} & 60.345 \\
\texttt{sanr200\_0\_9} & 200 & 2037 & 42 & 2500 & 43.856 & 74.4\% & \multicolumn{2}{c|}{$-$} & 49.274 \\ \hline
\end{tabular}
\caption{Results on the graphs from \cite{gaar2024different}.}
\label{table_gaarTable}
\end{table}

\begin{table}[]
\centering
\begin{tabular}{lccc|cc|r}
\hline
\multicolumn{1}{l}{\multirow{2}{*}{Graph   name}} & \multirow{2}{*}{$n$} & \multirow{2}{*}{$\abs*{E}$} & \multirow{2}{*}{$\alpha(G)$} & \multicolumn{2}{c|}{\lassColName{} bound   after} & \multirow{2}{*}{$\vartheta(G)$} \\
\multicolumn{1}{c}{} &  &  &  & \multicolumn{1}{l}{1 hour} & \multicolumn{1}{l|}{\hspace{1.8em} 4 hours} &  \\ \hline
\texttt{rand\_n200\_p002} & \multicolumn{1}{r}{200} & \multicolumn{1}{r}{407} & \multicolumn{1}{r|}{95} & \textbf{95.778} & \textbf{95.244} & \multicolumn{1}{r}{\textbf{95.778}} \\
\texttt{c\_fat200\_5} & \multicolumn{1}{r}{200} & \multicolumn{1}{r}{11427} & \multicolumn{1}{r|}{58} & 60.345 & 60.317 & \multicolumn{1}{r}{60.345} \\ \hline
\end{tabular}
\caption{Improved \lassColName{} bounds with longer ADMM running times.}
\label{table_improvedLassSDPName}
\end{table}

\subsubsection{SDP bounds on SageMath graphs}
We compute the \lassColName{} bounds for several graphs from the SageMath \cite{sagemath} software. Specifically, we consider SageMath graphs on at least 30 vertices,  for which $\vartheta(G)$ is strictly larger than $\alpha(G)$. 

The results are reported in \Cref{table_sageGraphs}, which uses the same column definitions as \Cref{table_gaarTable}.
The \lassColName{} bounds improve significantly over $\vartheta(G)$, and all of them equal $\alpha(G)$ when rounded down. 
On average, the ADMM algorithm required 1098 seconds to terminate for each graph.
It is worth noting that for each graph, the ADMM required at most 900 seconds to reach an iteration $\ell$ for which $\lfloor \lassObj{X^\ell} \rfloor = \alpha(G)$, see \eqref{eqn_lemmaShiftState}. The computation of BOUND~2 required at most 220 seconds, but there are some graphs in \Cref{table_sageGraphs} for which the floor of BOUND~2 does not equal $\alpha(G)$.
\begin{table}[ht]
\centering
\begin{tabular}{lrrr|rrr|rr|r}
\hline
\multicolumn{1}{l}{\multirow{2}{*}{Graph   name}} & \multicolumn{1}{c}{\multirow{2}{*}{$n$}} & \multicolumn{1}{c}{\multirow{2}{*}{$\abs*{E}$}} & \multicolumn{1}{c|}{\multirow{2}{*}{$\alpha(G)$}} & \multicolumn{3}{c|}{\lassColName{}} & \multicolumn{2}{c|}{BOUND 2 \cite{pucher2024class}} & \multicolumn{1}{c}{\multirow{2}{*}{$\vartheta(G)$}} \\
\multicolumn{1}{c}{} & \multicolumn{1}{c}{} & \multicolumn{1}{c}{} & \multicolumn{1}{c|}{} & \multicolumn{1}{l}{$\abs*{\mathcal{B}}$} & \multicolumn{1}{l}{Bound} & \multicolumn{1}{l|}{\colName{}} & \multicolumn{1}{l}{Bound} & \multicolumn{1}{l|}{\colName{}} & \multicolumn{1}{c}{} \\ \hline
\texttt{DoubleStarSnark} & 30 & 45 & 13 & 421 & \textbf{13.001} & 99.9\% & \textbf{13.001} & 99.8\% & \textbf{13.735} \\
\texttt{Wells} & 32 & 80 & 10 & 449 & \textbf{10.001} & 99.9\% & \textbf{10.906} & 54.7\% & \textbf{12.000} \\
\texttt{Sylvester} & 36 & 90 & 12 & 577 & \textbf{12.001} & 99.9\% & \textbf{12.034} & 97.7\% & 13.500 \\
\texttt{SzekeresSnark} & 50 & 75 & 21 & 1201 & \textbf{21.035} & 97.7\% & \textbf{21.359} & 76.6\% & 22.537 \\
\texttt{Klein3RegularGraph} & 56 & 84 & 23 & 1513 & \textbf{23.059} & 97.8\% & \textbf{23.908} & 67.4\% & 25.793 \\
\texttt{Gosset} & 56 & 756 & 4 & 841 & \textbf{4.048} & 97.0\% & 5.600 & 0.0\% & 5.600 \\
\texttt{Gritsenko} & 65 & 1040 & 6 & 1106 & \textbf{6.097} & 95.2\% & 8.062 & 0.0\% & 8.062 \\
\texttt{Meredith} & 70 & 140 & 34 & 2346 & \textbf{34.094} & 80.7\% & \textbf{34.003} & 99.3\% & \textbf{34.489} \\
\texttt{BrouwerHaemers} & 81 & 810 & 15 & 2500 & \textbf{15.041} & 99.3\% & 20.259 & 12.3\% & 21.000 \\
\texttt{HigmanSimsGraph} & 100 & 1100 & 22 & 2500 & \textbf{22.005} & 99.8\% & 26.667 & 0.0\% & 26.667 \\
\texttt{BiggsSmith} & 102 & 153 & 43 & 2500 & \textbf{43.536} & 85.4\% & 44.270 & 65.5\% & 46.686 \\
\texttt{Balaban11Cage} & 112 & 168 & 52 & 2500 & \textbf{52.124} & 92.2\% & \textbf{52.064} & 95.9\% & 53.595 \\ \hline
\end{tabular}
\caption{Results on several graphs from the SageMath software.}
\label{table_sageGraphs}
\end{table}

\section{\conclusionSectTitle{}}
\label{section_conclusion}
We have considered SDP bounds on the stability number of graphs obtained from the Lasserre hierarchy at relaxation level $2$, or relaxation levels intermediate to levels $1$ and $2$. Most of the SDPs considered here cannot be handled by IPMs, as they require excessive memory due to the large number of constraints. Therefore, we compute the bounds using the ADMM. The main operations of the ADMM algorithm are the projection onto the PSD cone \eqref{eqn_psdProjection}, and the projection onto half-spaces, see \Cref{lemma_halfSpaceProject}.
Although the former projection is significantly more computationally expensive than the latter, it is still manageable for the problem sizes considered in this work.
For improved performance, our ADMM algorithm also uses warm-starting.

For Lasserre levels intermediate to levels $1$ and $2$, we propose a method to select a basis of variables for the relaxation, see \Cref{section_choosingB}. We use this basis selection method to  choose  bases of size at most 2500 for computing bounds on $\alpha(G)$ for various graphs in \Cref{sect:numericBounds}. With this basis size, the ADMM algorithm often converges within one hour of computation time.

The computational experiments in \Cref{sect:numericBounds} show that the Lasserre hierarchy bounds, computed via the ADMM and referred to as \lassColName{} bounds, are competitive with other SDP-based stable set approaches from the literature, specifically \cite{gaar2020computational} and \cite{pucher2024class}. 
In particular, { to the best of our knowledge,} our approach provides the strongest SDP bounds on $\alpha(G)$ for a variety of graphs, see \Crefrange{table_jan15Results}{table_sageGraphs}. For some large and dense graphs in \Cref{table_gaarTable}, the bound from \cite{pucher2024class} cannot be computed using the IPM due to its large memory requirement, in contrast to 
our \lassColName{} bounds. 

As future work, it would be interesting to evaluate bounds from the intermediate level Lasserre hierarchy on the stability number of highly symmetric graphs. Specifically, our basis selection method from \Cref{section_choosingB} is currently not suited to exploit   symmetry in the underlying graph. 
{Preliminary computational experiments suggest that the current basis selection method can be  improved for highly symmetric graphs. In particular, the basis selection may benefit from taking the graph's automorphism group into account.}
Another future research direction is to investigate faster methods for projecting onto the PSD cone, which is the bottleneck of the ADMM. If we find an algorithm for projecting onto the PSD cone that is better suited for the ADMM than the standard \eigenDecomp{}, the ADMM could speed up significantly.

\subsection*{Statements and Declarations}
\textbf{Conflict of interest.} The authors declare that they have no conflict of interest. \\

\noindent\textbf{Funding.} No funds, grants, or other support was received. \\
 
\noindent\textbf{Availability of data and materials.} All the graphs used in this work are available online, or on request. \\

\noindent\textbf{Ethics approval and consent to participate.} Not applicable. \\

\noindent\textbf{Consent for publication.} Not applicable.\\

\noindent\textbf{Acknowledgements.} Not applicable.
\bibliographystyle{abbrvnat}
\bibliography{myRefs.bib}

\end{document}